\documentclass[12pt,BCOR=9.20mm,paper=a4]{scrartcl}
\usepackage{ayv_options}
\newcommand{\ID}[1]{\operatorname{LD}(#1)}
\newcommand{\SL}[1]{\operatorname{SL}(#1)}
\newcommand{\SC}[1]{\operatorname{SC}(#1)}
\newcommand{\VN}[2]{\mathscr{O}[#1,#2]}
\newcommand{\Lebesgue}{\mathscr{L}^1}
\usepackage{changes}
\numberwithin{equation}{section}

\title{Time-optimal problem in the space of probability measures}
\author{Yurii Averboukh}\address{HSE University, Moscow, Russia} \email{averboukh@gmail.com}
\author{Ekaterina Kolpakova}\address{Krasovskii Institute of Mathematics and Mechanics, \\ & Yekaterinburg, Russia}
		\email{eakolpakova@gmail.com}
\date{}

\begin{document}

	\maketitle
	\begin{abstract}
		This paper focuses on the value function in the time-optimal problem for a continuity equation in the space of probability measures. We derive the dynamic programming principle for this problem. Using this principle,  we prove that the Kruzhkov transform of the value function is a unique discontinuous viscosity solution to the corresponding Dirichlet problem for the Hamilton–Jacobi equation. Finally, we establish the 
		$\Gamma$-convergence of the value function in a perturbed problem to the value function in the unperturbed problem.
		\msccode{49J20, 49J52, 49L25, 49N70, 91A23, 49K20, 82C22}
	\end{abstract}
	

	\tableofcontents
	
	\section{Introduction}
In this paper, we study a control system in the space of probability measures, where the dynamics is governed by a nonlocal continuity equation driven by an external force. This equation models the evolution of a system comprising infinitely many identical particles, with each particle influenced by the collective field generated by all other particles.

The nonlocal continuity equation traces its origins to the pioneering work of Vlasov on plasma dynamics. Nowadays, it finds widespread application in various fields, including opinion dynamics models \cite{Carrillo2010, Naldi2010, Wang2022},  crowd behavior analysis \cite{Colombo2012, Dogbe2012, HUGHES2002507,Piccoli2018} and swarm dynamics~\cite{Fetecau2011, Mogilner1999}.

The objective of the control is to steer the system to a  target in minimal time. The target here is assumed to be an arbitrary closed set of probability measures. In this work, we develop the dynamic programming approach and investigate the properties of the value function.

The main results of the paper are as follows.
First, we establish the existence of an optimal relaxed control. Next, we derive the dynamic programming principle for the value function in this setting.
We then develop a framework based on the Dirichlet problem for the Hamilton-Jacobi equation, which, in this context, is a partial differential equation in the space of measures.  We demonstrate that the Kruzhkov transform of the value function satisfies the aforementioned Dirichlet problem in the viscosity sense.
Furthermore, by deriving a comparison principle, we prove the existence and uniqueness of a solution to the Dirichlet problem associated with the time-optimal problem for a nonlocal continuity equation. In general, the value function is discontinuous; however, we identify the sufficient condition that ensures its continuity.
Finally, we show that the values of the perturbed time-optimal problems 
$\Gamma$-converge to the solution of the limiting problem.

Now, let us give a brief overview of the existing literature and compare our results with what is already known. The study of control problems for the nonlocal continuity equation began with papers \cite{Brockett, Colombo}. Necessary optimality conditions in the form of the Pontryagin maximum principle for finite-horizon problems were derived in \cite{Bongini2017, Bonnet2019, Bonnet2021, Bonnet2019a, Pogodaev2016}.
We should also highlight the paper \cite{Badreddine2021}, where the authors develop the dynamic programming principle for the optimal control problem of the continuity equation on a finite time interval.  In that work, it is shown that the value function restricted to the space of compactly supported measures is a solution to the Cauchy problem for the Hamilton-Jacobi equation in the space of probability measures. Moreover, the uniqueness of the solution is also established  on the set of compactly supported  probability measures.

Notice that the problem studied in this paper assumes the system is influenced by an external force. Consequently, the velocity field is a continuous function of a point. An alternative approach considers the possibility of a discontinuous velocity field. This framework known  as the mean field type control theory is particularly relevant for systems involving intellectual agents striving to achieve a common goal.   For the current state of research in this direction, we refer to \cite{Averboukh,Jimenez2020,Jimenez2023}. Time-optimal problems for mean field type control systems were investigated in \cite{Cavagnari2017,Cavagnari2015,Cavagnari2021a}, where the target is assumed to be a hyperplane in the space of probability measures.  In those papers, the value function was characterized as a solution to a Dirichlet problem for the Hamilton-Jacobi equation in the space of measures. Furthermore, sufficient conditions for the regularity of the value function were established. However, these works do not provide a uniqueness result.

As previously mentioned, our approach to analyzing time-optimal problems relies on viscosity solutions to the Hamilton–Jacobi equation. This methodology was initially developed for finite-dimensional equations (see \cite{Bardi, crandall_lions, Subbotin}).

The concept of viscosity solutions has been extended in recent years to spaces of probability measures (see \cite{AussedatJerhaouiZidani2024,Badreddine2021, Cavagnari2015, Cavagnari2021a,cosso2024master,daudin2024comparison, Gangbo2019, Jimenez2020, Pham2018}). The definition of a viscosity solution inherently relies on the notions of subdifferentials and superdifferentials, which can themselves be introduced in several ways in the context of probability measures.
The central question here concerns the choice of admissible variations and the corresponding definition of sub- and superdifferentials. In our framework, admissible variations are Borel functions satisfying a sublinear growth condition. Consequently, elements of the sub- and superdifferentials belong to the dual class; specifically, they are functions whose product with any sublinearly growing variation is integrable. In general, these sub- and superdifferentials are larger than those in the approach originally proposed in \cite{Cardaliaguet_Quincampoix_2008}, where both admissible variations and elements of the sub-/superdifferentials are required to be square-integrable. Consequently, our definition of viscosity solution is slightly stronger than the one based solely on square-integrable variations. Nevertheless, the nonsmooth analysis developed here is particularly well-suited to the Dirichlet problem arising from time-optimal control of a nonlocal continuity equation. In particular, our proof of the comparison principle (Theorem~\ref{th1}) relies essentially on variations satisfying only a sublinear growth condition.

We now briefly review alternative approaches to defining sub-/superdifferentials and viscosity solutions for Hamilton-Jacobi equations in the Wasserstein space. One approach, employed in \cite{Badreddine2021} for finite-horizon optimal control problems with nonlocal continuity equations (see also \cite{Cardaliaguet_Quincampoix_2008}), considers variations of a measure along directions given by square-integrable functions. Another line of work uses directions defined by transport plans (see \cite{Ambrosio_Gangbo_2008, Cavagnari2015, Gangbo2019, Jimenez2020}); this is primarily applied to Hamilton-Jacobi equations associated with deterministic mean-field type control problems. In \cite{ambrosio}, elements of subdifferentials are even allowed to be distributions. {The paper~\cite{AussedatJerhaouiZidani2024} uses a definition of viscosity solution based on directionally differentiable test functions.} A recent contribution \cite{daudin2024comparison} introduces strong viscosity solutions for semilinear second-order Hamilton-Jacobi equations by leveraging entropy regularization.

All the approaches mentioned above are intrinsic, as they rely only on measures and directions within the space of measures. Alternatively, a probability measure with finite second moment can be represented as the push-forward of a square-integrable random variable, leading to an external approach to nonsmooth analysis in the space of probability measures (see \cite{Cardaliaguet2019, Gangbo2019, Pham2018,cosso2024master}). Connections between the external viewpoint and certain intrinsic definitions of sub- and superdifferentials are discussed in \cite{Cardaliaguet2019, Gangbo2019, Jimenez2024, Jimenez2023}. A comparative analysis of various intrinsic approaches can also be found in \cite{Jimenez2020, Jimenez2023}.

The paper is organized as follows. In Section~\ref{sect:prel}, we introduce the general notation used throughout the paper. In Section~\ref{sect:statement}, we formulate the time-optimal control problem for the nonlocal continuity equation, prove the existence of an optimal control, show that the value function is lower semicontinuous, and introduce the dynamic programming principle for the time-optimal mean-field-type control problem under consideration. Section~\ref{sect:dirichlet} introduces the Dirichlet problem for the Hamilton-Jacobi equation and defines the corresponding notion of a viscosity solution. The comparison principle for this problem is established in Section~\ref{sect:comparison}, which, in turn, implies the uniqueness of the viscosity solution. In Section~\ref{sect:characterization}, we prove that the Kruzhkov transform of the value function is a viscosity solution to the aforementioned Dirichlet problem. Additionally, we derive a sufficient condition ensuring the continuity of the value function and show that the values of perturbed time-optimal control problems $\Gamma$-converge to the original problem (see Section~\ref{sect:topological}). Finally, Section~\ref{sect:example} provides an example of a time-optimal control problem for the controlled continuity equation.

\section{Preliminaries}\label{sect:prel}
If $X_1,\ldots, X_n$ are sets, $i_1,\ldots, i_k\in \{1,\ldots,n\}$, then $\operatorname{p}^{i_1,\ldots,i_k}$ denotes a natural projection of $X_1\times\ldots\times X_n$ onto $X_{i_1}\times\ldots\times X_{i_k}$. Additionally, $\operatorname{Id}$ stands for the identity mapping.

If $(\Omega',\mathcal{F}')$ and $(\Omega'',\mathcal{F}'')$ are measurable spaces, $m$ is a measure on $\mathcal{F}'$, and $h: \Omega'\rightarrow \Omega''$ is a $\mathcal{F}'/\mathcal{F}''$-measurable map, then we denote by $h\sharp m$ the push-forward measure defined by the rule: for each $\Gamma\in \mathcal{F}''$,
$$ (h\sharp m)(\Gamma)=m(h^{-1}(\Gamma)).$$

If $(X,\rho_X)$ is a Polish space, $x\in X$, and $r>0$, then $\mathbb{B}_r(x)$ denotes the ball of radius~$r$ centered at $x$. For a closed set \(M\), \(\operatorname{dist}(x,M)\) denotes the distance from a point \(x\in X\) to the set \(M\):
\begin{equation*} \operatorname{dist}(x,M)\triangleq \inf\big\{\rho_X(x,y):\, y\in M\big\}.\end{equation*} 

Furthermore, $\mathcal{M}(X)$ denotes the space of finite positive Borel measures on $X$, while $\mathcal{P}(X)$ stands for the space of Borel probability measures on $X$, i.e., $\mathcal{P}(X)=\{m\in\mathcal{M}(X):\, m(X)=1\}$.

We consider the narrow convergence on $\mathcal{M}(X)$. This means that $\{m_n\}_{n=1}^\infty\subset \mathcal{M}(X)$ converges to $m\in\mathcal{M}(X)$ if, for every $\phi\in C_b(X)$, 
\begin{equation*} \int_X\phi(x)m_n(dx)\rightarrow \int_X\phi(x)m(dx)\text{ as }n\rightarrow\infty.\end{equation*}  Obviously, $\mathcal{P}(X)$ is closed w.r.t. the narrow convergence.

We denote the space of probability measures with the finite second moment by $\mathcal{P}^2(X)$, i.e., $m\in\mathcal{P}(X)$ belongs to $\mathcal{P}^2(X)$ provided that, for some $x_*\in X$, $\int_X\rho_X^2(x,x_*)m(dx)<\infty$. If $X$ is a Banach space, then we denote
\begin{equation*} \varsigma(m)\triangleq \Bigg[\int_X\|x\|^2m(dx)\Bigg]^{1/2}.\end{equation*} 
We endow $\mathcal{P}^2(X)$ with the second Wasserstein distance defined by the rule: $$W_2(m_1,m_2)=\Bigg[\inf\limits_{\pi\in\Pi(m_1,m_2)}\int_{X\times X}\big(\rho_X(x_1,x_2)\big)^2\pi(d(x_1,x_2))\Bigg]^{1/2},$$
where   $\Pi(m_1,m_2) $ stands for the set of probabilities $\pi\in \mathcal{P}^2(X\times X)$ such that $\operatorname{p}^i\sharp \pi=m_i$.

{If, in addition, \(Y\) is a Banach space, then \(L^0(X;Y)\) denotes the set of all Borel measurable functions from \(X\) to \(Y\). Let \(m\) be a measure on \(X\), let \(w:X\to[0,+\infty)\) be a Borel measurable function, and let \(p\geq 1\). Then \(L^p_w(X,m;Y)\) is the space of Borel measurable functions \(\varphi:X\to Y\) such that
	\begin{equation*}
		\int_X w(x)\|\varphi(x)\|^p_Y\,m(dx)<\infty.
	\end{equation*}
	Furthermore, \(L^\infty_w(X,m;Y)\) denotes the space of Borel measurable functions \(\varphi:X\to Y\) such that \(w\varphi\) is essentially bounded with respect to \(m\). In the spaces \(L^p_w(X,m;Y)\), \(p\in[1,+\infty]\), we identify, as usual, functions coinciding \(m\)-a.e.
	Recall that the space \(L^\infty_w(X,m;Y)\) is usually equipped with the norm \(\|\cdot\|_{L^\infty_w(X,m;Y)}\) defined by the rule: \(\|\varphi\|_{L^\infty_w(X,m;Y)}\triangleq \inf\{C>0:\, w(x)\|\varphi(x)\|_Y\leq C\text{ for }m\text{-a.e. }x\in X\}\).
	If \(w\equiv 1\), we omit the subscripts in the notation \(L^p\), \(p\in[1,+\infty]\).}

If $s_1,s_2:\rd\rightarrow\rd$ are such that the mapping $x\mapsto \langle s_1(x), s_2(x)\rangle$ lies in $L^1(\rd,m;\mathbb{R})$, then we put
\begin{equation*} \langle s_1,s_2\rangle_m\triangleq \int_{\rd} \langle s_1(x),s_2(x)\rangle m(dx).\end{equation*}  Notice that  $\langle\cdot,\cdot\rangle_m$ is the inner product on $L^2(\rd,m;\rd)$.

{Let \(\mathscr{a}:\rd\to[1,+\infty)\) denote the regularized norm defined by
	\begin{equation}\label{intro:almost_norm}
		\mathscr{a}(x)\triangleq (1+\|x\|^2)^{1/2}.
	\end{equation}
	The space \(L^\infty_{1/\mathscr{a}}(\rd,m;\rd)\) of all  functions \(F:\rd\to\rd\) with sublinear growth (equivalently, such that \(F/\mathscr{a}\) is \(m\)-essentially bounded) plays a crucial role in the paper.}

{In what follows, we denote by \(\SL{m}\) the space \(L^\infty_{1/\mathscr{a}}(\rd,m;\rd)\) with a redefined convergence structure: a sequence \(\{F_n\}_{n=1}^\infty\subset\SL{m}\) converges to \(F\in\SL{m}\) provided
	\begin{itemize}
		\item there exists a constant \(C>0\) such that
		\[
		\|F_n(x)\|\leq C\mathscr{a}(x)\quad\text{for }m\text{-a.e. }x\in\rd;
		\]
		\item \(F_n(x)\to F(x)\) as \(n\to\infty\) for \(m\)-a.e. \(x\in\rd\).
	\end{itemize}
	Further, we say that a family \(\{F_h\}_{h\in(0,\varepsilon)}\subset\SL{m}\) converges to \(F\in\SL{m}\) in \(\SL{m}\) if the sequence \(\{F_{h_n}\}_{n=1}^\infty\) converges to \(F\) for every sequence \(\{h_n\}_{n=1}^\infty\subset(0,\varepsilon)\) tending to zero.}

{Notice that \(\SL{m}\) is a sequential convergence vector space in the sense of~\cite[\S1.7]{Beattie2002}; that is, we endow the set of functions satisfying the sublinear growth condition with the convergence structure defined by sequences, so that addition and scalar multiplication are continuous. Moreover, the convergence defined on \(\SL{m}\) is weaker than the one generated by the norm \(\|\cdot\|_{L^\infty_{1/\mathscr{a}}(\rd,m;\rd)}\).}

{One can consider the dual space of \(\SL{m}\), i.e., the space of continuous linear functionals from \(\SL{m}\) to \(\mathbb{R}\) \cite[\S4.1]{Beattie2002}. The representation of this space is discussed in Appendix~\ref{sect:app_dual}. There, we show (see Corollary~\ref{dual:corollary:SL}) that the dual space of \(\SL{m}\) is isomorphic to the weighted Lebesgue space \(L^1_{\mathscr{a}}(\rd,m;\rd)\), i.e., the space of  functions \(s:\rd\to\rd\) such that
	\begin{equation*}
		\int_{\rd} \|s(x)\|\,\mathscr{a}(x)\,m(dx)<\infty.
	\end{equation*}
	For brevity, we use the notation
	\begin{equation*}
		\ID{m}\triangleq L^1_{\mathscr{a}}(\rd,m;\rd).
	\end{equation*}
	The action of the  functional corresponding to \(s\in \ID{m}\) is given, for \(F\in\SL{m}\), by the pairing \(\langle s,F\rangle_m\). }

We denote the Lebesgue measure on \([0,+\infty)\) by \(\Lebesgue\), and its restriction to \([0,T]\) by \(\Lebesgue|_T\).

In what follows, we fix  a compact metric space $U$ that plays the role of the control space. Let 
\(\mathcal{U}\) be the set of measures \(\xi\)  on $[0,+\infty) \times U$ such that \(\operatorname{p}^1\sharp\xi=\Lebesgue\). Elements of  $\mathcal{U}$ are regarded as  relaxed (generalized) controls on \([0,+\infty)\). {We will use also relaxed controls on \([0,T]\). The set of such controls is denoted by \(\mathcal{U}_T\), i.e., \(\mathcal{U}_T\) is the set of measures \(\eta\) on \([0,T]\times U\) satisfying \(\operatorname{p}^1\sharp\eta=\Lebesgue|_T\).  Notice that 
	\(\mathcal{U}_T\subset\mathcal{M}([0,T]\times U)\) and it is closed w.r.t. the narrow convergence. }

If $\xi\in\mathcal{U}$, $T>0$, then $\xi|_T$ stands for its restriction  on $[0,T]\times U$. Notice that this operation provides a natural projection of \(\mathcal{U}\) on \(\mathcal{U}_T\). {Moreover, this allows us to introduce on \(\mathcal{U}\) the topology of local narrow convergence, defined as follows.}
\begin{definition}\label{prel:def:convergence_U} We say that a sequence $\{\xi_n\}_{n=1}^\infty\subset \mathcal{U}$ converges to $\xi\in\mathcal{U}$ provided that, for each $T>0$, $\{\xi_n|_T\}_{n=1}^\infty$ converges to $\{\xi|_T\}$ narrowly.
\end{definition} 
{	Notice that in this definition it suffices to consider only natural numbers \(T\). Indeed, if
	\(\{\xi_n|_T\}_{n=1}^\infty\) converges narrowly, then
	\(\{\xi_n|_{T'}\}_{n=1}^\infty\) converges narrowly for every
	\(T'\leq T\). Since each set \(\mathcal{U}_T\), \(T\in\mathbb{N}\), is compact, by Prokhorov's
	theorem, a standard diagonal argument shows that \(\mathcal{U}\) is compact.}

Sometimes, we will work with instantaneous controls. They can be introduced through the disintegration. For each \(T>0\) and \(\eta\in\mathcal{U}_T\), \cite[Theorem 10.4.14]{Bogachev2007} gives that  one can find a family of probability measures  $\{\eta(\cdot|t)\}_{t\in [0,T]}\subset \mathcal{P}(U)$ such that, for every function $\phi\in L^1([0,T]\times U,\eta;\mathbb{R})$,
\begin{equation*}
	\label{dis}
	\int_{[0,T]\times U} \phi(t,u)\eta(d(t,u))=\int_0^T\int_{U} \phi(t,u)\eta(du|t)dt.
\end{equation*} Letting in this equality \(\eta=\xi|_T\) for \(\xi\in\mathcal{U}\) and natural \(T\) and passing to the limit when \(T\rightarrow\infty\), we obtain that 
there exists a family of probability measures  $\{\xi(\cdot|t)\}_{t\in [0,+\infty)}\subset \mathcal{P}(U)$ such that, for every  locally integrable w.r.t. the measure \(\xi\) function  $\phi:[0,+\infty)\times U\rightarrow\mathbb{R}$,
\begin{equation}
	\label{dis}
	\int_{[0,+\infty)\times U} \phi(t,u)\xi(d(t,u))=\int_0^{+\infty}\int_{U} \phi(t,u)\xi(du|t)dt.
\end{equation}

Moreover, this family is  $\Lebesgue$-a.e. unique. 
If $\{\xi(\cdot|t)\}_{t\in [0,+\infty)}$ is a family of probabilities on $U$ that is weakly measurable, one can uniquely construct a measure $\xi\in\mathcal{U}$ such that $\operatorname{p}^1\sharp\xi=\Lebesgue$ and~\eqref{dis} holds true. 

Moreover, each measurable function $u:[0,+\infty)\rightarrow U$ generates an element of $\mathcal{U}$ such that $\xi(\cdot|t)\triangleq \delta_{u(t)}(\cdot)$. The set of such elements is dense in $\mathcal{U}$ \cite[Theorem IV.3.10]{Warga}. 

\section{Statement of the problem}\label{sect:statement}
The central object of this work is the controlled nonlocal continuity equation given by:
\begin{equation} \label{dyn}
	\partial_t m(t)+\operatorname{div}(f(x,m(t),u(t))m(t))=0, \ \ t\in [0,+\infty).
\end{equation}
Here, $x\in \mathbb{R}^d$, $m(t)\in \mathcal{P}^2(\mathbb{R}^d)$,  $u(t)\in U$, and { \(U\) is a metric compact set}. We interpret $u(t)$ as an instantaneous control. Thus, $U$ is the space of admissible controls. 

Equation~\eqref{dyn} describes the evolution of the distribution of infinitely many similar particles  obeying  the following equation
$$ \frac{d}{dt}x(t)=f(x(t),m(t),u(t)), \ {x(t)}\in \mathbb{R}^d,  \ u(t)\in U.$$ Here, the distribution of agents at the instant $t$ is denoted $m(t)\in  \mathcal{P}^2(\mathbb{R}^d)$, while $u(t)$ stands for an external control that influences the whole system. 

We impose the following assumptions.
\begin{hypothesis}\label{hyp1} The function $f$ is continuous in all arguments. \end{hypothesis}
\begin{hypothesis}\label{hyp2} There exists a constant $L > 0 $ such that, for every $x_1,x_2\in \mathbb{R}^d$, $m_1,m_2\in \mathcal{P}^2(\mathbb{R}^d)$,  $u\in U$, one has the inequality $$ \|f(x_1,m_1,u)-f(x_2,m_2,u)\|\leq L(\|x_1-x_2\|+W_2(m_1,m_2)).$$ 
\end{hypothesis}
\begin{hypothesis}\label{hyp3} There exists a constant $C_f > 0 $ such that, for every $x\in \mathbb{R}^d$, $m\in \mathcal{P}^2(\mathbb{R}^d)$,  $u\in U$, one has the inequality $$ \|f(x,m,u)\|\leq C_f\Bigg(\mathscr{a}(x)+\int_{\rd}\mathscr{a}(x)m(dx)\Bigg),$$ 
\end{hypothesis} where {the function \(\mathscr{a}:\rd\rightarrow (0,+\infty)\) is defined by~\eqref{intro:almost_norm}}.

{\begin{remark}
		Hypothesis~\ref{hyp2} implies that  the following sublinear growth condition:
		\begin{equation*} 
			\|f(t,x,m,u)\|\leq C'(1+\|x\|+\varsigma(x)).
		\end{equation*} 
		It is slightly weaker than the one imposed in Hypothesis~\ref{hyp3}. Indeed, the function  
		\begin{equation*} 
			f(t,x,m,u)\triangleq W_2(m,\delta_0),
		\end{equation*} 
		where \(\delta_0\) denotes the Dirac measure concentrated at the origin, satisfies Hypothesis~\ref{hyp2} but violates Hypothesis~\ref{hyp3}.\end{remark}}

In the paper, we use relaxed controls. As  mentioned above, each relaxed control is a measure on $[0,+\infty)\times U$ with the marginal distribution on $[0,+\infty)$ coinciding with the Lebesgue measure. Furthermore, given $\xi\in \mathcal{U}$, there exists its disintegration w.r.t. the Lebesgue measure denoted by $\xi(\cdot|t)$. Formally, the usage of  relaxed controls means that we replace continuity equation~\eqref{dyn} with

\begin{equation} \label{dyn_relaxed}
	\partial_t m(t)+\operatorname{div}\Bigg(\int_U f(x,m(t),u)\xi(du|t)\cdot m(t)\Bigg)=0, \ \ t\in [0,+\infty).
\end{equation} We will consider this equation in the distributional sense.

\begin{definition}\label{def1}
	Let $T>0$, $\xi\in\mathcal{U}_T$. A measure-valued function $m(\cdot): [0, T]\rightarrow \mathcal{P}^2(\mathbb{R}^d)$ is called a solution of equation \eqref{dyn_relaxed} on $[0,T]$ if, for every $\varphi\in C_0^1((0,T)\times\mathbb{R}^d)$,
	it holds that
	$$ \int_{[0,T]\times U}\int_{\mathbb{R}^d} \partial_t \varphi(t,x) + \left\langle \nabla \varphi(t,x), f(x,m(t),u)\right\rangle m(t,dx)\xi(d(t,u))=0.$$
	If $\xi\in\mathcal{U}$, then a measure-valued function $m(\cdot):[0,+\infty)\rightarrow\mathcal{P}^2(\mathbb{R}^d)$ is a solution of \eqref{dyn_relaxed} on $[0,+\infty)$ if, for every $T>0$, the restriction  of $m(\cdot)$ on $[0,T]$ is a solution of~\eqref{dyn_relaxed} on $[0,T]$. 
\end{definition} 

Given $\mu\in\mathcal{P}^2(\mathbb{R}^d)$, $\xi\in\mathcal{U}$, we denote by $m(\cdot;\mu,\xi)$ the   unique solution of~\eqref{dyn_relaxed}  on  $[0,+\infty)$ satisfying $m(0)=\mu$.   If $\xi\in\mathcal{U}_T$, we will, by a slight abuse of notation, also denote the solution of~\eqref{dyn_relaxed} on $[0,T]$ satisfying $m(0)=\mu$ by $m(\cdot;\mu,\xi)$.

An alternative definition of the solution to continuity equation~\eqref{dyn_relaxed} relies on the particle interpretation. To introduce it, for given $y\in \mathbb{R}^d $, a continuous function $m(\cdot):[0,T]\rightarrow \mathcal{P}^2(\mathbb{R}^d )$, $\xi\in \mathcal{U}_T$, we denote by $X(\cdot;y,m(\cdot),\xi)$ the solution of the initial value problem
\begin{equation*}\label{dyn2}
	\frac{d}{dt}x(t)=\int_{U} f(x(t),m(t),u)\xi(du|t), \ x(0)=y.
\end{equation*}
The function $X(\cdot;y,m(\cdot),\xi)$ describes the motion of a particle on the time interval $[0,T]$ in the case where the distribution of all particles is given by $m(\cdot)$, while the system is affected by a relaxed control $\xi$. The motion on $[0,+\infty)$ can be defined in the same way. 

From  {\cite[Proposition 8.1.8]{ambrosio}}, a measure-valued function $m(\cdot):[0,T]\rightarrow\mathcal{P}^2(\mathbb{R}^d)$ solves~\eqref{dyn_relaxed} for some $\xi\in\mathcal{U}$ if and only if 
\begin{equation*} m(t)=X(t;\cdot,m(\cdot),\xi)\sharp m(0).\end{equation*} 

Now, let us describe the time-optimal problem examined in the paper.
We assume that we are provided with  a closed set $M\subset \mathcal{P}^2(\mathbb{R}^d)$. Furthermore, we denote  $G\triangleq\mathcal{P}^2(\mathbb{R}^d)\setminus M$. Additionally,  $\partial G$ stands for the border of $G$. Let us consider the  functional $\tau$ defined on $C([0,+\infty);\mathcal{P}^2(\mathbb{R}^d))$ by
\begin{equation}  \label{funcl}
	\tau(m(\cdot))\triangleq \inf\big\{t\in[0;+\infty): \ m(t)\in M\big\}.
\end{equation}
Notice that, if, for every $t>0$, $m(t)\not\in M$, then $\tau(m(\cdot))=+\infty$. The time-optimal problem means that, given an initial distribution $\mu\in\mathcal{P}^2(\mathbb{R}^d)$, one wishes to minimize the quantity $\tau(m(\cdot;\mu,\xi))$. Thus, the value function of  time-optimal problem \eqref{dyn}, \eqref{funcl} is defined by the rule: 
\begin{equation*} \label{val}
	\operatorname{Val}(\mu)=\inf\big\{\tau(m(\cdot;\mu,\xi)):\ \ \xi\in\mathcal{U}\big\}.
\end{equation*}

Notice that while this definition formally involves the relaxed continuity equation~\eqref{dyn_relaxed}, the value function can be equivalently defined using usual controls, relaxation of the target set, and the dynamics~\eqref{dyn2}, as shown in Remark~\ref{remark:value} below.

\subsection{Properties of the value function}
The properties of the value function in a time-optimal problem  in $\mathbb{R}^d$ are well studied (see~\cite{Bardi,Subbotin} for a comprehensive review).
It is  known that the value function is lower semicontinuous and satisfies the dynamic programming principle. In this section, we derive the analogous properties for the examined time-optimal control on the space of probability measures.

First, let us recall the stability of the trajectories on each finite interval proved in~\cite{Averboukh23}.
\begin{proposition}\label{lm0}
	For each natural  $n$, let $f_n$, 
	be a  function from $\mathbb{R}^d\times\mathcal{P}^2(\mathbb{R}^d)\times U$ to $\mathbb{R}^d$, and let $\mu_n\in\mathcal{P}^2(\mathbb{R}^d)$, $\xi_n\in \mathcal{U}$. Assume, additionally that
	\begin{itemize}
		\item each function $f_n$ is continuous w.r.t. all variables and Lipschitz continuous w.r.t. $x$ and $m$ for the constant $L$;
		\item there exists a continuous function $f:\mathbb{R}^d\times\mathcal{P}^2(\mathbb{R}^d)\times U\rightarrow \mathbb{R}^d$ that is Lipschitz continuous w.r.t. $x$ and $m$ such that, for each $c>0$, 
		\begin{equation*}
			\begin{split}
				\sup\Big\{\|f_n(x,&m,u)-f(x,m,u)\|:\\ &x\in\rd,\, m\in\mathcal{P}^2(\rd),\, \varsigma(m)\leq c,\, u\in U\Big\}\rightarrow 0\text{ as }n\rightarrow\infty;
			\end{split}
		\end{equation*}
		\item the sequences $\{\mu_n\}_{n=1}^\infty$    
		and $\{\xi_n\}_{n=1}^\infty$ converge to $\mu\in\mathcal{P}^2(\mathbb{R}^d)$ and $\xi\in\mathcal{U}$ respectively;
		\item for each natural $n$, $m_n(\cdot)$ is the solution of~\eqref{dyn_relaxed} for the dynamics equal to $f_n$, relaxed control $\xi_n$ such that $m_n(0)=\mu_n$.
	\end{itemize} Then, for every $T>0$,   the sequence of  $\{m_n(\cdot)\}_{n=1}^{\infty}$ converges to
	$m(\cdot;\mu,\xi)$ in $C([0,T];\mathcal{P}^2(\mathbb{R}^d))$.
\end{proposition}

Now, we establish the existence of an optimal control whenever the value function is finite.

\begin{theorem}\label{th:optima_existence}
	Assume that $\mu\in \operatorname{cl}G $ satisfies $\operatorname{Val}(\mu)<+\infty$. Then, there exists a relaxed control $\xi_*\in \mathcal{U}$ such that 
	\begin{equation*} \operatorname{Val}(\mu)=\tau(m(\cdot;\mu,\xi_*)).\end{equation*} 
\end{theorem}
\begin{remark}
	We say that  a relaxed control $\xi_*\in \mathcal{U}$ is optimal if  
	\begin{equation*} \operatorname{Val}(\mu)=\tau(m(\cdot;\mu,\xi_*)).\end{equation*} 
\end{remark}
\begin{proof}[Proof of Theorem~\ref{th:optima_existence}]
	By the definition of the value function, there exists a sequence $\{\xi_n\}_{n=1}^\infty\subset \mathcal{U}$ such that 
	\begin{equation*}\label{equaliyu:Val_tau}\operatorname{Val}(\mu)=\theta_*\triangleq\lim_{n\rightarrow\infty}\tau(m(\cdot;\mu,\xi_n)).\end{equation*} 
	Furthermore, we denote $\theta_n\triangleq \tau(m(\cdot;\mu,\xi_n)). $ Since $\mathcal{U}$ is compact, without loss of generality, one may assume that $\{\xi_n\}_{n=1}^\infty$ converges to some control $\xi_*\in\mathcal{U}$. {Proposition~\ref{lm0} implies that
		\(\{m(\theta_n;\mu,\xi_n)\}_{n=1}^\infty\) converges to \(m(\theta_*;\mu,\xi_*)\). Since \(m(\theta_n;\mu,\xi_n)\in M\) and \(M\) is closed, the inequality
		\begin{equation*}
			\operatorname{dist}(m(\theta_*;\mu,\xi_*),M)
			\leq W_2\bigl(m(\theta_*;\mu,\xi_*),m(\theta_n;\mu,\xi_n)\bigr)
			+\operatorname{dist}(m(\theta_n;\mu,\xi_n),M)
		\end{equation*}
		yields that \(m(\theta_*;\mu,\xi_*)\in M\).} Hence, $\operatorname{Val}(\mu)=\tau(m(\cdot;\mu,\xi_*))$.
\end{proof}

\begin{proposition}\label{pr1}
	The value function for problem \eqref{dyn}, \eqref{funcl} is lower semicontinuous. 
\end{proposition}
\begin{proof}
	Let $\{\mu_k\}_{k=1}^\infty\subset \operatorname{cl} G $  be a sequence of initial conditions for problem \eqref{dyn} converging to $\mu_0\in  \operatorname{cl} G$. Denote $T_k\triangleq \operatorname{Val}(\mu_k)$. First,  we consider the case where $\{T_k\}_{k=1}^\infty$ is bounded. Without loss of generality, we assume that the sequence $\{T_k\}_{k=1}^\infty$ converges to some $T_0\in [0,+\infty)$.
	Furthermore, due to Theorem~\ref{th:optima_existence}, for each $k$, there exists a  relaxed control $\xi_k$ such that  $T_k= \tau(m(\cdot;\mu_k,\xi_k))$.  Again, without loss of generality, we assume that the whole sequence $\{\xi_k\}_{k=1}^\infty$ converges to $\xi_0\in \mathcal{U}$.   Proposition \ref{lm0} gives that the probabilities $m(T_k;\mu_k,\xi_k)$ tend to $m(T_0;\mu_0,\xi_0)$ as $k\rightarrow\infty$.
	Since the set $M$ is closed, while $m(T_{k};\mu_{k},\xi_{k})\in M$, we have that $m(T_0;\mu_0, \xi_0)\in M$. Thus, using the choice of $T_k$, we conclude that 
	\begin{equation*} \lim_{k\rightarrow\infty}\operatorname{Val}(\mu_k)=T_0\geq \tau(m(\cdot;\mu_0,\xi_0))\geq \operatorname{Val}(\mu_0).\end{equation*} 
	
	If the limit of \(\{T_k\}_{k=1}^\infty\) is infinite,   one obviously has that 
	\begin{equation*} \lim_{k\rightarrow\infty}\operatorname{Val}(\mu_k)\geq \operatorname{Val}(\mu_0).
	\end{equation*} 
	{}
\end{proof}

{We complete this section with the representation of the value function using usual controls.
	\begin{proposition}\label{remark:value} Let, for \(\varepsilon>0\), \begin{equation*} M^\varepsilon\triangleq \operatorname{cl}\Big(\big\{m\in\mathcal{P}^2(\rd):\, \operatorname{dist}(m,M)\leq \varepsilon\big\}\Big).\end{equation*}   Furthermore, let   $\tau^\varepsilon(\cdot)$ assign to a trajectory $m(\cdot)\in C([0,+\infty);\mathcal{P}^2(\rd))$ the quantity $\inf\{t\in [0,+\infty):\, m(t)\in {M^\varepsilon}\}$.
		Then, 
		\begin{equation*}
			\begin{split}
				\operatorname{Val}(\mu)=\lim_{\varepsilon\downarrow 0}\inf\Big\{\tau^\varepsilon(m(\cdot)):\, m(\cdot)\text{ satisfies }\partial_t m(t)+\operatorname{div}(f(&x,m(t),u(t))m(t))=0,\\ m(0)=\mu&,\ \ u(\cdot)\in L^0([0,+\infty);U)\Big\}.\end{split}
		\end{equation*}		
	\end{proposition}
	\begin{proof}
		Given \(\varepsilon>0\), denote
		\begin{equation*}
			\begin{split}
				\mathcal{V}^\varepsilon(\mu)=\inf\Big\{\tau^\varepsilon(m(\cdot)):\, m(\cdot)\text{ satisfies }\partial_t m(t)+\operatorname{div}(f(&x,m(t),u(t))m(t))=0,\\ m(0)=\mu&,\ \ u(\cdot)\in L^0([0,+\infty);U)\Big\}.
			\end{split}
		\end{equation*}
		Notice that, for each \(\mu\in\mathcal{P}(\rd)\), the mapping \((0,+\infty)\ni\varepsilon\mapsto \mathcal{V}^{\varepsilon}(\mu)\) is non-increasing.
		\newline
		For every \(T<\operatorname{Val}(\mu)\), one can find \(\hat{\varepsilon}\) such that
		\begin{equation*} 
			\operatorname{dist}(m(t,\mu,\xi),M)\geq 2\hat{\varepsilon},\quad t\in [0,T],\ \xi\in\mathcal{U}.
		\end{equation*} 
		Owing to the embedding of the usual controls into the set of relaxed controls, and since \(\mathcal{V}^\varepsilon(\mu)\) is non-increasing as a function of \(\varepsilon\), it follows that, for \(\varepsilon\leq \hat{\varepsilon}\), we have the inequality $
		\mathcal{V}^\varepsilon(\mu)\geq T$.
		Passing sequentially to the limits as \(\varepsilon\downarrow 0\) and \(T\uparrow \operatorname{Val}(\mu)\), we conclude that
		\begin{equation*} 
			\lim_{\varepsilon\downarrow 0}\mathcal{V}^\varepsilon(\mu)\geq \operatorname{Val}(\mu).
		\end{equation*} 
		To prove the reverse inequality, we take \(\xi\in\mathcal{U}\) such that \(\operatorname{Val}(\mu)=\tau(m(\cdot),\mu,\xi)\). Since the set of usual controls embeds densely in \(\mathcal{U}\), there exists a sequence \(\{u_n(\cdot)\}_{n=1}^\infty\subset L^0([0,+\infty),U)\) such that the sequence \(\{\xi_n\}_{n=1}^\infty\), where \(\xi_n(\cdot|t)=\delta_{u_n(t)}(\cdot)\) for \(t\in [0,t)\), converges to \(\xi\). In view of Proposition~\ref{lm0}, there exists a sequence \(\{\varepsilon_n\}_{n=1}^\infty\) with \(\varepsilon_n\rightarrow 0\) as \(n\rightarrow \infty\) and, for each natural \(n\),
		\begin{equation*}
			m(\operatorname{Val}(\mu),\mu,\xi_n)\in M^{\varepsilon_n}. \end{equation*} 
		This implies that $
		\mathcal{V}^{\varepsilon_n}(\mu)\leq \operatorname{Val}(\mu)$.
		Passing to the limit as \(n\rightarrow 0\){ and using the fact that the mapping \(\varepsilon\mapsto \mathcal{V}^{\varepsilon}(\mu)\) is non-descreasing}, we deduce that
		\begin{equation*} 
			\lim_{\varepsilon\downarrow 0}\mathcal{V}^\varepsilon(\mu)\leq \operatorname{Val}(\mu).
		\end{equation*} {}
\end{proof}	}

\subsection{Dynamic programming principle}

\begin{theorem}\label{th:dynprog}
	The value function for the problem \eqref{dyn}, \eqref{funcl}  satisfies the dynamic programming principle on $\operatorname{cl}G$, i.e., for every $\mu\in \operatorname{cl} G$ and $h\in [0, \operatorname{Val}(\mu)]$,
	$$  \operatorname{Val}(\mu)-h=\inf_{\xi\in \mathcal{U}_{h}} \operatorname{Val}(m(h;\mu,\xi)).$$
\end{theorem}
\begin{proof} If $\xi\in\mathcal{U}_h$, $\eta\in\mathcal{U}$, then we denote by $\xi\diamond_h\eta$ the relaxed control determined by its disintegration as follows:
	\begin{equation*} (\xi\diamond_h\eta)(\cdot|t)\triangleq\begin{cases}
			\xi(\cdot|t), & t\in [0,h),\\
			\eta(\cdot|t-h), & t\in [h,+\infty).
	\end{cases}\end{equation*} 
	We have that
	\begin{equation*}
		\begin{split}
			\inf\limits_{\xi\in \mathcal{U}_h}\operatorname{Val} (m(h;\mu,\xi))&=\inf\limits_{\xi\in \mathcal{U}_{h}}\inf\limits_{\eta\in \mathcal{U}} \tau(m(\cdot;m(h;\mu,\xi),\eta))\\
			&=\inf\limits_{\xi\in \mathcal{U}_{h}}\inf\limits_{\eta\in \mathcal{U}}\tau(m(\cdot;\mu,\xi\diamond_h\eta))-h\\
			&=\inf\limits_{\xi\in \mathcal{U}}\tau(m(\cdot;\mu,\xi))-h=\operatorname{Val}(\mu)-h.
		\end{split}
	\end{equation*}
	{}
\end{proof}

\section{The Dirichlet problem for the Hamilton-Jacobi equation}\label{sect:dirichlet}
In this section, given $m\in\mathcal{P}^2(\rd)$, we  use the set of  functions satisfying sublinear growth condition $\SL{m}$ and its dual $\ID{m}$ introduced in Section~\ref{sect:prel}. 

We define the  Hamiltonian by the following rule: for  $m\in\mathcal{P}^2(\mathbb{R}^d)$ and $s\in \ID{m}$,
$$H(m,s)=\inf\limits_{u\in U} \int_{\mathbb{R}^d}\left\langle s(x), f(x,m,u)\right\rangle m(dx).$$

The value function for problem \eqref{dyn}--\eqref{funcl} can take infinite values.
So, we apply the Kruzhkov transform and define the function $\phi:\operatorname{cl}G\rightarrow \mathbb{R}$ by the rule: \begin{equation}\label{intro:Kruzhkov}
	\phi(m)=1-e^{-\operatorname{Val}(m)}.
\end{equation} 
Notice that $\phi$ takes values in $[0,1]$. Thus, we arrive at the following  Dirichlet problem for the Hamilton-Jacobi equation and the transformed value function:
\begin{equation} \label{eq:hj}
	H(m,\nabla\phi(m))+1-\phi(m)=0, \, m\in G;
\end{equation}
\begin{equation}\label{bound:hj}
	\phi(m)=0,\, m\in \partial G.
\end{equation}
Here, we formally use $\nabla\phi(m)$ to designate the derivative of the function $\phi$. 
Below, we develop a viscosity solution approach based on nonsmooth analysis in the Wasserstein space, utilizing only sublinear variations.

{First, given \(m\in\mathcal{P}^2(\rd)\) and \(F\in\SL{m}\), we denote by \(\VN{m}{F}\) the set of all sequences
	\(\{(h_n,F_n)\}_{n=1}^\infty\subset(0,+\infty)\times\SL{m}\)
	such that \(h_n\to0\) and \(F_n\to F\) in \(\SL{m}\) as \(n\to\infty\).}

Let $m\in \mathcal{P}^2(\rd)$, $r>0$, and let $\psi$ be a function defined on~$\mathbb{B}_r(m)$ with values in $\mathbb{R}\cup\{+\infty,-\infty\}$ such that $\psi(m)\in\mathbb{R}$.
\begin{definition}\label{def:sub_super_differential}
	A subdifferential of the function $\psi$ at the measure $m$ consists of all functions $s\in \ID{m}$ such that, for every $F\in \SL{m}$  {and \(\{(h_n,F_n)\}_{n=1}^\infty\in \VN{m}{F}\)}, 
	\begin{equation}\label{intro:subdiff}
		{\liminf_{n\rightarrow\infty} \frac{\psi((\operatorname{Id}+h_nF_n)\sharp m)-\psi(m)- h_n\langle s,F\rangle_m}{h_n}\geq 0.}
	\end{equation} We denote the subdifferential of the function $\psi$ at the measure $m$ by {$\partial^-\psi(m)$}. 
	
	Similarly, the superdifferential of the $\psi$ at the measure $m$ denoted by {$\partial^+\psi(m)$} contains all functions $s\in \ID{m}$ satisfying the following condition: for $F\in \SL{m}$ {and \(\{(h_n,F_n)\}_{n=1}^\infty\in \VN{m}{F}\)},
	\begin{equation}\label{intro:superdiff}
		{\limsup_{n\rightarrow\infty} \frac{\psi((\operatorname{Id}+h_nF_n)\sharp m)-\psi(m)- h_n\langle s,F\rangle_m}{h_n}\leq 0.}
	\end{equation}
\end{definition}
\begin{remark}\label{remark:sub_super_diff_eq}
	An alternative definition of the sub-/superdifferentials involves the concept of Hadamard directional derivatives.  Let a measure $m$, a radius $r$ and a function $\psi$ be as above. A Hadamard lower directional derivative of the function $\psi$ at the probability $m$ in a direction $F\in \SL{m}$ is 
	\begin{equation*}  {d^-_H\psi(m;F)\triangleq\inf\Bigg\{\liminf_{n\rightarrow\infty} \frac{\psi((\operatorname{Id}+h_nF_n)\sharp m)-\psi(m)}{h_n}:\, \{(h_n,F_n)\}_{n=1}^\infty\in \VN{m}{F}\Bigg\} .} \end{equation*} 
	Similarly, a Hadamard upper directional derivative of the function $\psi$ at a probability $m\in G$ in a direction $F\in \SL{m}$ is equal to \begin{equation*}  {d^+_H\psi(m;F)\triangleq\sup\Bigg\{\limsup_{n\rightarrow\infty}\frac{\psi((\operatorname{\operatorname{Id}}+h_nF_n)\sharp m)-\psi(m)}{h_n}:\, \{(h_n,F_n)\}_{n=1}^\infty\in \VN{m}{F}\Bigg\} .} \end{equation*} 
	One can easily show that 
	\begin{equation*} \partial^-\psi(m)=\big\{s\in \ID{m}:\, \left\langle s,F\right\rangle_m \leq d^-_H\psi(m;F){\text{ for every }F\in\SL{m}}\big\},\end{equation*} 
	\begin{equation*} \partial^+\psi(m)=\big\{s\in \ID{m}:\, \left\langle s,F\right\rangle_m \geq d^+_H\psi(m;F){\text{ for every }F\in\SL{m}}\big\}.\end{equation*}   
\end{remark}

\begin{remark}\label{remark:sum:subdiff}
	Notice that from the very definition of the subdifferential (see~\eqref{intro:subdiff}) it directly follows that, if $\psi_1,\psi_2: \mathbb{B}_r(m)\rightarrow\mathbb{R}\cup\{+\infty,-\infty\}$ for some positive number $r$ are such that $\psi_1(m)$ and $\psi_2(m)$ are finite, then
	\begin{equation*}\label{incl:subdiff_sum}
		\partial^-\psi_1(m)+\partial^-\psi_2(m)\subset \partial^-(\psi_1+\psi_2)(m).
	\end{equation*} Analogously,~\eqref{intro:superdiff} implies the inclusion
	\begin{equation*}
		\partial^+\psi_1(m)+\partial^+\psi_2(m)\subset \partial^+(\psi_1+\psi_2)(m).
	\end{equation*}
\end{remark}
\begin{remark}\label{remark:null} As in the finite dimensional case, one has that, if $m\in\mathcal{P}^2(\rd)$ and $\psi:\mathbb{B}_r(m)\rightarrow \mathbb{R}\cup \{+\infty\}$ are such that $\psi(m)<\infty$ and $m$ is a local minimizer for the function $\psi$, then 
	\begin{equation*} \mathbf{0}\in \partial^-\psi(m);\end{equation*}  hereinafter $\mathbf{0}$ stands for the function from $\rd$ to $\rd$ equal to zero everywhere.
\end{remark}

Our approach to the definition of the viscosity solution to the Dirichlet problem is close to one proposed in \cite[Definition 18.3, Theorem 18.6]{Subbotin} and \cite[\S IV.3]{Bardi} for the finite dimensional case.  In particular, we postulate that the viscosity solution coincides with the supersolution. Notice that in \cite{Bardi} this notion is called an envelope viscosity solution, while in \cite{Subbotin} the words `minimax solution' are used.

\begin{definition}\label{def:supersolution}
	A bounded  lower semicontinuous function $\phi:\operatorname{cl} G\rightarrow \mathbb{R}$ is called a viscosity supersolution of the Hamilton-Jacobi equation \eqref{eq:hj}   provided that,
	for every $m\in G$ and $s\in \partial^-\phi(m)$, $$H(m,s)+1-\phi(m)\leq 0.$$ We say that a bounded lower semicontinuous function $\phi:\operatorname{cl} G\rightarrow \mathbb{R}$ is called a viscosity supersolution of the Dirichlet problem~\eqref{eq:hj},~\eqref{bound:hj} if it is a viscosity supersolution of \eqref{eq:hj} and satisfies~\eqref{bound:hj}. 
\end{definition}
\begin{definition}\label{def:subsolution}
	A bounded  upper semicontinuous function $\phi:\operatorname{cl} G\rightarrow \mathbb{R}$ is called a viscosity subsolution of the Hamilton-Jacobi equation \eqref{eq:hj}   if, for every $m\in G$ and $s\in \partial^+\phi(m)$, one has that
	$$H(m,s)+1-\phi(m)\geq 0.$$
	We say that a bounded upper semicontinuous function $\phi:\operatorname{cl} G\rightarrow \mathbb{R}$ is called a viscosity subsolution of the Dirichlet problem \eqref{eq:hj},~\eqref{bound:hj} if it is a viscosity subsolution of \eqref{eq:hj} and satisfies~\eqref{bound:hj}. 
\end{definition}
\begin{remark} If $\phi$ is a nonnegative viscosity subsolution of \eqref{eq:hj},~\eqref{bound:hj}, then it is continuous at every point of $\partial G$.
\end{remark}

\begin{definition}\label{def:solution}
	A bounded lower semicontinuous function \(\phi:\operatorname{cl}G\to\mathbb{R}\) is called a viscosity solution of problem \eqref{eq:hj},~\eqref{bound:hj} if
	\begin{itemize}
		\item $\phi$  is a supersolution of  problem~\eqref{eq:hj},~\eqref{bound:hj};
		\item there exists a sequence of functions $\{\phi_k\}_{k=1}^\infty$ defined on $\operatorname{cl}G$ with values in $\mathbb{R}$ such that each function $\phi_k$ is a subsolution of \eqref{eq:hj}, \eqref{bound:hj} while, for every $m\in G$, one has that 
		\begin{equation*} \phi(m)=\lim_{k\rightarrow\infty}\phi_k(m).\end{equation*} \end{itemize}
\end{definition}

Let us conclude this section with a discussion of the proposed definition of a viscosity solution. Our approach is based on sub- and superdifferentials. A key feature of our definition is that we restrict admissible variations to those satisfying a sublinear growth condition. Consequently, elements of the sub- and superdifferentials are permitted to grow faster than square-integrable functions. This differs from the frameworks in \cite[Chapter 10]{ambrosio} and \cite{Badreddine2021,Cardaliaguet_Quincampoix_2008,Jimenez2020,Jimenez2023}, where variations and elements of sub-/superdifferentials are required to be square-integrable (or, more generally, where variations belong to $L^p$
and sub-/superdifferentials to $L^q$).


Our motivation is twofold. First, in the proof of the comparison principle (Theorem~\ref{th1}), we rely heavily on a functional whose sub-/superdifferential contains an element that grows faster than an $L^2$
function. Second, when verifying that the Kruzhkov transform of the value function is a viscosity solution of~\eqref{eq:hj},~\eqref{bound:hj} (see Theorem~\ref{th:character}), we  use only variations of sublinear growth. Hence, if the controlled continuity equation satisfies the standard sublinear growth condition, the requirement of square-integrable variations is unnecessary.

The notion of viscosity solution given in Definition~\ref{def:solution} appears stronger than those following \cite{Badreddine2021,Cardaliaguet_Quincampoix_2008}, where sub-/superdifferentials are defined via square-integrable variations. Whether one can define a viscosity solution to the Dirichlet problem \eqref{eq:hj},~\eqref{bound:hj} while admitting general  $L^2$
variations and still preserving the two main results (the comparison principle and the property that the Kruzhkov transform of the value function satisfies~\eqref{eq:hj},~\eqref{bound:hj}) remains an intriguing open question. Additionally, it would be interesting to develop an approach based on strong viscosity solutions in the spirit of \cite{Crandall1985,Jimenez2020,Jimenez2023}.

\section{Comparison principle}\label{sect:comparison}
\begin{theorem}\label{th1}
	Let $\phi_1:\operatorname{cl}G\rightarrow \mathbb{R}$ be a subsolution of  problem~\eqref{eq:hj} and let $\phi_2:\operatorname{cl}G\rightarrow \mathbb{R}$ be
	a supersolution of  problem~\eqref{eq:hj}. Assume that $\phi_1\leq \phi_2$ on $\partial G$. Then, for each  $m\in \operatorname{cl}G$, $\phi_1(m)\leq \phi_2(m).$ 
\end{theorem}
Before we proceed with the proof of the comparison principle, let us formulate the uniqueness result.

\begin{corollary}\label{corollary:uniqueness} There exists at most one solution of the Dirichlet problem~\eqref{eq:hj}, \eqref{bound:hj}.
\end{corollary}
This statement directly follows from		Theorem~\ref{th1} and Definition~\ref{def:solution}.

The proof of the comparison principle relies on the differentiability properties of the squared Wasserstein distance  as well as the  properties of a moment gauge function that will be used to regularize the difference  $\phi_1-\phi_2$. 

Since we use the specific form of the sub- and superdifferentials, we need the following auxiliary statement. 
{\begin{lemma}\label{lm:theta_superddiff} Let \(m,m_b\in \mathcal{P}^2(\rd)\), $\vartheta[m,m_b]\in \Pi(m,m_b)$ be an optimal plan between $m$ and $m_b$. Define a barycenter of the plan $\vartheta$  by the rule:
		\begin{equation}\label{intro:hat_vartheta}
			\hat{\vartheta}[m,m_b](x)\triangleq \int_{\rd}(x-y)\vartheta[m,m_b](dy|x).
		\end{equation} Then \(\hat{\vartheta}[m,m_b]\) lies in the superdifferential  of the function $m'\mapsto\frac{1}{2}W_2^2(m',m_b)$ at $m$ in the sense of Definition~\ref{def:sub_super_differential}.
\end{lemma}}
\begin{proof}
	{Due to the Jensen's inequality and the definition of disintegration, we have  
		\begin{equation}\label{ineq:norm_vartheta}
			\begin{split}
				\int_{\rd}\|\hat{\vartheta}[m,m_b](x)\|^2m(dx)&=\int_{\rd}\bigg\|\int_{\rd}(x-y)\vartheta[m,m_b](dy|x)\bigg\|^2m(dx)\\&\leq \int_{\rd\times\rd}\|x-y\|^2\vartheta[m,m_b](d(x,y))<\infty.\end{split}\end{equation}
		Furthermore, \(\hat{\vartheta}[m,m_b]\in\ID{m}\). Indeed, \eqref{ineq:norm_vartheta} and H\"older's inequality yield
		\begin{equation*} \begin{split}
				\int_{\rd}\|\hat{\vartheta}[m,&m_b](x)\|\mathscr{a}(x)m(dx)\\ &\leq \Bigg[\int_{\rd}\|\hat{\vartheta}[m,m_b](x)\|^2m(dx) \Bigg]^{1/2}\cdot \Bigg[\int_{\rd}\mathscr{a}^2(x)m(dx)\Bigg]^{1/2}. \end{split}\end{equation*}  }
	
	{Further, recall from \cite[Theorem 10.2.2]{ambrosio} that, given $m'\in \mathcal{P}^2(\rd)$, $\pi\in\Pi(m,m')$ and $\varpi\in\mathcal{P}((\rd)^3)$ such that $\operatorname{p}^{1,2}\sharp\varpi=\pi$ and $\operatorname{p}^{1,3}\sharp\varpi=\vartheta$, we have 
		\begin{equation*}\begin{split}\frac{1}{2}W^2_2(m',m_b)-\frac{1}{2}W^2_2(m,m_b)-\int_{(\rd)^3}\langle x-y,&x'-x\rangle\varpi(d(x,x',y))\\&\leq \frac{1}{2}\int_{\rd\times\rd}\|x'-x\|^2\pi(d(x,x')).\end{split}\end{equation*}   
		Let \(F\in \SL{m}\) and let \(\{(h_n,F_n)\}_{n=1}^\infty\in \VN{m}{F}\). Applying the previous inequality with $m'=(\operatorname{Id}+h_nF_n)\sharp m$, yields  
		\begin{equation*}
			\begin{split}
				\frac{1}{2}W^2_2((\operatorname{Id}+h_nF_n)\sharp m,m_b)-\frac{1}{2}W^2_2(m,m_b)-h_n&\langle \hat{\vartheta}[m,m_b],F_n\rangle_m\\&\leq  {\frac{1}{2}\int_{\rd}\|F_n(x)\|^2m(dx)\cdot h_n^2}.
			\end{split}
	\end{equation*} }

	{Since \(\|F_n(x)\|\leq C\mathscr{a}(x)\) for all \(x\) and all \(n\), with a constant \(C>0\) independent of \(n\), while
		\(
		\langle\hat{\vartheta}[m,m_b],F_n\rangle_m\to \langle\hat{\vartheta}[m,m_b],F\rangle_m,
		\)
		we obtain the statement of the lemma.}
\end{proof}

Now let us discuss the moment gauge function, which plays a crucial role in the proof of the comparison principle. 
\begin{lemma}\label{B:lm:g} Let $m\in\mathcal{P}^2(\rd)$. Then, there exists a continuously differentiable function $g:(0,+\infty)\rightarrow (0,+\infty)$ such that 
	\begin{enumerate}[label=(G\arabic*)]
		\item\label{B:cond:G_positive} $g$ is positive, monotonously increasing on $[1,+\infty)$ with $g(r)\rightarrow +\infty$ as $r\rightarrow +\infty$;
		\item\label{B:cond:G_int} $\int_{\rd}\mathscr{a}^2(x)g(\mathscr{a}(x))m(dx)<\infty$;
		\item\label{B:cond:G_derivative} $rg'(r)\leq g(r)$, where $g'$ stands for the derivative of the function $g$.
	\end{enumerate}
\end{lemma}

\begin{lemma}\label{B:lm:derivative_g} Let $m\in \mathcal{P}^2(\rd)$, and let $g:(0,+\infty)\rightarrow (0,+\infty)$ satisfy conditions \ref{B:cond:G_positive}--\ref{B:cond:G_derivative}. Denote
	\begin{equation}\label{intro:V_g}V_g(m)\triangleq \int_{\rd}\mathscr{a}^2(x)g(\mathscr{a}(x))m(dx).\end{equation}
	Then, the function $\mathscr{j}_g(\cdot)$ defined by the rule:
	\begin{equation*} \mathscr{j}_g(x)\triangleq  \big[2g(\mathscr{a}(x))+\mathscr{a}(x)g'(\mathscr{a}(x))\big]x\end{equation*}  lies in the intersection of $\partial^-V_g(m)$ and $\partial^+V_g(m)$.
\end{lemma}
These auxiliary statements are proved in Appendix~\ref{sect:B}. 

{Note that condition~\ref{B:cond:G_int} is stronger than second-moment finiteness, which is already ensured by \(m\in\mathcal{P}^2(\rd)\). Moreover, this, together with property~\ref{B:cond:G_positive}, implies that the sublevel sets of the potential energy functional \(V_g\) are compact (see~\cite[Lemma 5.1.7, Proposition 7.1.5]{ambrosio}).} Furthermore, this functional is differentiable in the sense that $\mathscr{j}_g$ belongs to both the sub- and superdifferentials whenever $V_g$ is finite.

\begin{proof}[Proof of Theorem~\ref{th1}] We modify  the proof of \cite[Theorem V.1.3]{Bardi}. 
	As there, we will argue by contradiction. This means that there exist $\hat{m}\in G$ and $\varkappa>0$ such that 
	\begin{equation}\label{ineq:varkappa}
		\phi_1(\hat m)-\phi_2(\hat m)\geq 2\varkappa>0.\end{equation} {Here, we take into account the fact that, by assumption, \(\phi_1\leq \phi_2\) on \(\partial G\).} Let $g$ be a function satisfying conditions~\ref{B:cond:G_positive}--\ref{B:cond:G_derivative} for $m=\hat{m}$.
	For each $\varepsilon\in (0,1)$, we define the function $\Phi_\varepsilon:(\operatorname{cl} G)^2\rightarrow\mathbb{R}$ as follows: 
	\begin{equation*} \Phi_\varepsilon(m_1,m_2)\triangleq \phi_1(m_1)-\phi_2(m_2)-\frac{1}{2\varepsilon}W_2^2(m_1,m_2)-\beta\big[V_g(m_1)\big]^\kappa-\beta\big[V_g(m_2)\big]^\kappa,\end{equation*}  where $V_g$ is defined by~\eqref{intro:V_g},
	$\kappa\triangleq \min\{1,1/(6C_f)\}$, $C_f$ is a constant from Hypothesis~\ref{hyp3}, {$\beta>0$} is such that 
	\begin{equation}\label{ineq:phi_1_2_kappa}\phi_1(\hat{m})-\phi_2(\hat{m})-2\beta\big[V_g(\hat{m})\big]^\kappa\geq \varkappa>0.\end{equation} This implies that, for every $\varepsilon>0$,
	\begin{equation*}\label{ineq:Phi_varepsilon_delta}
		\sup_{(m_1,m_2)\in(\operatorname{cl}G)^2}\Phi_\varepsilon(m_1,m_2)\geq \Phi_\varepsilon(\hat{m},\hat{m})\geq \varkappa.
	\end{equation*}

	Notice that the function $\Phi_\varepsilon$ is u.s.c. Indeed, $\phi_1$ is u.s.c. by Definition~\ref{def:subsolution}, while $\phi_2$ is l.s.c. by Definition~\ref{def:supersolution}. Furthermore, the Wasserstein distance is continuous and the function $V_g$ is l.s.c. according to~\cite[Lemma 5.1.7]{ambrosio}. Moreover, the function $\Phi_\varepsilon$ is bounded above.
	
	Let $\hat{c}$ be a common upper bound for the values $|\phi_1(m)|$, $|\phi_2(m)|$, $m\in\operatorname{cl} G$, and let
	\begin{equation*} C_1\triangleq \big[2\hat{c}\beta^{-1}\big]^{1/\kappa}.\end{equation*}  Notice that, if $V_g(m_1)\vee V_g(m_2)> C_1$, $\varepsilon>0$, then
	\begin{equation*} \Phi_\varepsilon(m_1,m_2)< 0.\end{equation*}  Set \begin{equation*} \mathcal{K}\triangleq \{m\in\operatorname{cl} G:\, V_g(m)\leq C_1\}.\end{equation*}  Direct calculations show that the set $\mathcal{K}$ is tight and $2$-uniformly integrable. By~\cite[Proposition 7.1.5]{ambrosio} and the lower semicontinuity  of $V_g$, the set $\mathcal{K}$ is compact. Moreover, {$\hat{m}\in \mathcal{K}$}. 
	
	Since $\Phi_\varepsilon$ is u.s.c. and $\mathcal{K}$ is compact, the function $\Phi_\varepsilon$ admits a maximizer $(m_1^\varepsilon,m_2^\varepsilon)$ on $\mathcal{K}^2$. Moreover, on $(\operatorname{cl} G)^2\setminus (\mathcal{K})^2$, the function $\Phi_\varepsilon$ is negative, while at $\Phi_\varepsilon(\hat{m},\hat{m})\geq \varkappa$. Thus, the maximum of $\Phi_\varepsilon$ on $(\operatorname{cl} G)^2$ is attained at $(m_1^\varepsilon, m_2^\varepsilon)$. Moreover,
	\begin{equation}\label{ineq:Phi_hat_m}
		\Phi_\varepsilon(m_1^\varepsilon,m_2^\varepsilon)\geq \Phi_\varepsilon(\hat{m},\hat{m})\geq \varkappa.
	\end{equation}

	From the latter inequality, it follows that $W_2(m_1^{\varepsilon},m_2^{\varepsilon})\rightarrow 0$ as $\varepsilon\rightarrow 0$. We wish to prove that
	\begin{equation}\label{convergence:dist_eps}\frac{W_2^2(m_1^{\varepsilon},m_2^{\varepsilon})}{2\varepsilon}\rightarrow 0\text{ as }\varepsilon\rightarrow 0.\end{equation} To this end, we denote 
	\begin{equation*} A\triangleq \sup_{m\in \mathcal{K}} \big[\phi_1(m)-\phi_2(m)-2\beta [V_g(m)]^\kappa\big].\end{equation*}  We use the fact that the maximum of $\Phi_\varepsilon$ is attained at $(m_1^\varepsilon,m_2^\varepsilon)$ and observe that 
	\begin{equation}\label{ineq:A_A}\begin{split}
			A\leq \Phi_\varepsilon(m_1^{\varepsilon},m_2^{\varepsilon})\leq  \phi_1(m_1^{\varepsilon})-\phi_2({m_2^{\varepsilon}})-\beta \big[V_g(m_1^{\varepsilon})\big]^\kappa-\beta\big[V_g(m_2^{\varepsilon})\big]^\kappa.\end{split}\end{equation}
	{Now let us consider a vanishing sequence \(\{\varepsilon_n\}_{n=1}^\infty\).  
		Since each pair \((m_1^{\varepsilon_n},m_2^{\varepsilon_n})\) belongs to the compact set \(\mathcal{K}^2\), and since \(W_2(m_1^{\varepsilon_n},m_2^{\varepsilon_n})\to 0\) as \(n\to\infty\),  passing to a subsequence if necessary, we may assume that the sequences \(\{m_1^{\varepsilon_n}\}_{n=1}^\infty\) and \(\{m_2^{\varepsilon_n}\}_{n=1}^\infty\) converge to the same measure \(m^0\).
		Now recall that 	$\phi_1$ is u.s.c., while $\phi_2$ and $V_g$ are l.s.c. Thus,  
		\begin{equation*} \begin{split}
				\lim_{n\rightarrow\infty} \Bigg[\phi_1(m_1^{\varepsilon_n})-\phi_2(m_2^{\varepsilon_n})-\beta \big[V_g(&m_1^{\varepsilon_n})\big]^\kappa-\beta\big[V_g(m_2^{\varepsilon_n})\big]^\kappa\Bigg]\\ &
				\leq \phi_1(m^0)-\phi_2(m^0)-2\beta \big[V_g(m^0)\big]^\kappa \leq A.\end{split}\end{equation*} } From this and~\eqref{ineq:A_A}, we obtain~\eqref{convergence:dist_eps}.

	Now we  consider two cases.
	\begin{enumerate}[label=(\roman*)]
		\item\label{comparison:case:boundary} There exists a sequence $\{\varepsilon_k\}_{k=1}^\infty$ converging to zero such that $(m_1^{0},m_2^{0})=\lim_{k\rightarrow\infty}(m_1^{\varepsilon_k},m_2^{\varepsilon_k})\in \partial(G\times G)$;
		\item\label{comparison:case:interior} For every sequence $\{\varepsilon_k\}_{k=1}^\infty$ converging to zero such that  $(m_1^{\varepsilon_k},m_2^{\varepsilon_k})\rightarrow (m_1^{0},m_2^{0})$ as $k\rightarrow\infty$ one has that 
		\((m_1^{0},m_2^{0})\in \operatorname{int}(G\times G)\).
	\end{enumerate}
	In  case \ref{comparison:case:boundary}, we pass to the limit in inequality~\eqref{ineq:Phi_hat_m}. Thanks to~\eqref{convergence:dist_eps}, we deduce that $m_1^{0}=m_2^{0}\in\partial G$, while 
	\begin{equation*} \begin{split}
			\limsup_{k\rightarrow\infty} \Phi_{\varepsilon_k}(m_1^{\varepsilon_k},m_2^{\varepsilon_k})\leq \phi_1(m_1^{0})-\phi_2(m_1^{0}).\end{split}\end{equation*}  Here we used additionally the upper semicontinuity of the function $\phi_1$ and the lower semicontinuity of the function $\phi_2$. Thanks to the boundary conditions of these functions and~\eqref{ineq:phi_1_2_kappa}, we obtain
	\begin{equation*} \varkappa \leq \limsup_{k\rightarrow\infty}\Phi_{\varepsilon_k}(m_1^{\varepsilon_k},m_2^{\varepsilon_k})\leq 0.\end{equation*}  This contradicts the choice of $\varkappa$.
	
	In case~\ref{comparison:case:interior}, we consider only those $\varepsilon$ for which $(m_1^{\varepsilon},m_2^{\varepsilon})\in \operatorname{int}(G\times G)$ and set
	\begin{equation}\label{comparison:intro:psi_1}
		\begin{split}
			\psi_{1,\varepsilon}(m_1)\triangleq \Phi_\varepsilon(m_1,m_2^{\varepsilon})=
			\phi_1(m_1)-\phi_2(&m_2^{\varepsilon})-\frac{1}{2\varepsilon}W_2^2(m_1,m_2^{\varepsilon})\\&-\beta \big[V_g(m_1)\big]^\kappa-\beta \big[V_g(m_2^{\varepsilon})\big]^\kappa;
	\end{split}\end{equation}
	\begin{equation}\label{comparison:intro:psi_2}\begin{split}
			\psi_{2,\varepsilon}(m_2)\triangleq -\Phi_\varepsilon(m_1^{\varepsilon},m_2)=
			-\phi_1(m_1^{\varepsilon})+\phi_2(&m_2)+\frac{1}{2\varepsilon}W_2^2(m_1^{\varepsilon},m_2)\\&+\beta \big[V_g(m_1^{\varepsilon})\big]^\kappa+\beta \big[V_g(m_2)\big]^\kappa.
	\end{split}\end{equation} Since the pair $(m_1^\varepsilon,m_2^\varepsilon)$ provides the maximum of $\Phi_\varepsilon$ on $(\operatorname{cl} G)^2$, the function $\psi_{1,\varepsilon}$ attains a maximum at $m_1^{\varepsilon}$. Thus (see Remark~\ref{remark:null}), we have that $\mathbf{0}\in\partial^+\psi_{1,\varepsilon}(m_1^{\varepsilon})$. Let $\vartheta[m_1^{\varepsilon},m_2^{\varepsilon}]$ be an optimal plan between $m_1^{\varepsilon}$ and $m_2^{\varepsilon}$. Recall (see~\eqref{intro:hat_vartheta}) that 
	\begin{equation*} \hat{\vartheta}[m_1^\varepsilon,m_2^\varepsilon](x_1)\triangleq \int_{\rd}(x_1-x_2)\vartheta[m_1^\varepsilon,m_2^\varepsilon](dx_2|x_1).\end{equation*} 
	Thanks to Lemma~\ref{lm:theta_superddiff}, $\hat{\vartheta}[m_1^\varepsilon,m_2^\varepsilon]$ lies in the superdifferential of the mapping $m_1\mapsto \frac{1}{2}W_2^{{2}}(m_1,m_2^{\varepsilon})$ at $m_1^{\varepsilon}$.
	
	{Now let us consider a function \begin{equation*} s_1^{\varepsilon}\triangleq \varepsilon^{-1}\hat{\vartheta}[m_1^{\varepsilon},m_2^{\varepsilon}]+\kappa\beta\big[V_g(m_1^{\varepsilon})\big]^{\kappa-1}\cdot j_g.\end{equation*}  First, from Remark~\ref{remark:sum:subdiff} and Lemma~\ref{B:lm:derivative_g}, we deduce that \(s_1^{\varepsilon}\) lies in the superdifferential of the function \(m_1\mapsto \phi_2(m_2^{\varepsilon})+\frac{1}{2\varepsilon}W_2^2(m_1,m_2^{\varepsilon})+\beta \big[V_g(m_1)\big]^\kappa+\beta \big[V_g(m_2^{\varepsilon})\big]^\kappa\) at \(m_1^\varepsilon\). Furthermore, we express \(\phi_1(m)\) from~\eqref{comparison:intro:psi_1}. Since \(\mathbf{0}\in \psi_{1,\varepsilon}(m_1^\varepsilon)\), using once again Remark~\ref{remark:sum:subdiff}, we conclude
		\begin{equation*}
			s_1^\varepsilon\in\partial^+\phi_1(m_1^{\varepsilon}).
	\end{equation*}} Now recall that $\phi_1$ is a viscosity subsolution of~\eqref{eq:hj}. Thus, we have that 
	\begin{equation}\label{ineq:H_phi_1}
		H(m_1^{\varepsilon},s_1^{\varepsilon})+1-\phi_1(m_1^{\varepsilon})\geq 0.
	\end{equation} Analogously, if 
	$\hat{\vartheta}[m_2^\varepsilon,m_1^\varepsilon]$ is defined by the rule 
	\begin{equation*} \hat{\vartheta}[m_2^\varepsilon,m_1^\varepsilon](x_2)\triangleq \int_{\rd}(x_2-x_1)\vartheta[m_1^\varepsilon,m_2^\varepsilon](dx_1|x_2),\end{equation*}  then
	\begin{equation*} s_2^{\varepsilon}\triangleq -\varepsilon^{-1}\hat{\vartheta}[m_2^{\varepsilon},m_1^{\varepsilon}]-\kappa\beta\big[V_g(m_2^{\varepsilon})\big]^{\kappa-1}\cdot j_g\in\partial^-\phi_2(m_2^{\varepsilon}).\end{equation*}  This can be obtained as above from Remark~\ref{remark:sum:subdiff}, Lemmas~\ref{lm:theta_superddiff},~\ref{B:lm:derivative_g} and equality~\eqref{comparison:intro:psi_2}. 
	
	Recall that $\phi_2$ is a viscosity supersolution of~\eqref{eq:hj}. Hence,
	\begin{equation*} H(m_2^{\varepsilon},s_2^{\varepsilon})+1-\phi_2(m_2^{\varepsilon})\leq 0.\end{equation*}  This and~\eqref{ineq:H_phi_1} imply the following inequality
	\begin{equation}\label{ineq:phi_1_phi_2} \phi_1(m_1^{\varepsilon})-\phi_2(m_2^{\varepsilon})\leq H(m_1^{\varepsilon},s_1^{\varepsilon})-H(m_2^{\varepsilon},s_2^{\varepsilon}).\end{equation}
	Furthermore, let $u_\varepsilon$ be such that 
	\begin{equation*} 
		H(m_2^{\varepsilon},s_2^{\varepsilon})=\min_{u}\langle s_2^{\varepsilon},f(\cdot,m_2^{\varepsilon},u)\rangle_{ m_2^{\varepsilon}}
		=\langle s_2^{\varepsilon},f(\cdot,m_2^{\varepsilon},u_\varepsilon)\rangle_{m_2^{\varepsilon}}
	\end{equation*}  Therefore,
	\begin{equation}\label{ineq:Ham_m_1,m_2}
		\begin{split}
			H(m_1^{\varepsilon},s_1^{\varepsilon}&)-H(m_2^{\varepsilon},s_2^{\varepsilon})\\\leq 
			\varepsilon^{-1}\langle &\hat{\vartheta}[m_1^\varepsilon,m_2^\varepsilon],f(\cdot,m_1^{\varepsilon},u_\varepsilon)\rangle_{m_1^{\varepsilon}}
			+\varepsilon^{-1}\langle \hat{\vartheta}[m_2^\varepsilon,m_1^\varepsilon],f(\cdot,m_2^{\varepsilon},u_\varepsilon)\rangle_{m_2^{\varepsilon}}\\+
			\kappa&\beta\big[V_g(m_1^{\varepsilon})\big]^{\kappa-1} \langle j_g,f(\cdot,m_1^{\varepsilon},u_\varepsilon)\rangle_{m_1^{\varepsilon}}+\kappa\beta\big[V_g(m_2^{\varepsilon})\big]^{\kappa-1}\langle j_g,f(\cdot,m_2^{\varepsilon},u_\varepsilon)\rangle_{m_2^{\varepsilon}}.
		\end{split}
	\end{equation}
	Now let us  evaluate the right-hand side of this inequality.
	First, we use Hypothesis~\ref{hyp2} and obtain
	\begin{equation}\label{ineq:Ham_fisrt_term}
		\begin{split}
			\varepsilon^{-1}\langle &\hat{\vartheta}[m_1^\varepsilon,m_2^\varepsilon],f(\cdot,m_1^{\varepsilon},u_\varepsilon)\rangle_{m_1^{\varepsilon}}
			+\varepsilon^{-1}\langle \hat{\vartheta}[m_2^\varepsilon,m_1^\varepsilon],f(\cdot,m_2^{\varepsilon},u_\varepsilon)\rangle_{m_2^{\varepsilon}}\\&=
			\varepsilon^{-1}\int_{(\rd)^2}\langle x_1-x_2, f(x_1,m_1^{{\varepsilon}},u_\varepsilon)-f(x_2,m_2^{{\varepsilon}},u_\varepsilon)\rangle \vartheta[m_1^{\varepsilon},m_2^{\varepsilon}](d(x_1,x_2))\\ &\leq 
			2L\varepsilon^{-1}W_2^2(m_1^{{\varepsilon}},m_2^{{\varepsilon}}).
		\end{split}
	\end{equation} 
	Furthermore, using the definition of the function $j_g$, condition~\ref{B:cond:G_derivative} and Hypothesis~\ref{hyp3}, we conclude that
	\begin{equation*} \big|\langle j_g,f(\cdot,m_1^{\varepsilon},u_\varepsilon)\rangle_{m_1^{\varepsilon}}\big|
		\leq 3C_f\int_{\rd} \mathscr{a}(x)g(\mathscr{a}(x))\Bigg[\mathscr{a}(x)+\int_{\rd}\mathscr{a}(y)m_1^{\varepsilon}(dy)\Bigg]m_1^{\varepsilon}(dx).
	\end{equation*}  
	Inequality  \cite[(5)]{gurland1967inequality} and the monotonicity of the function $g$ yield that 
	\begin{equation}\label{ineq:Ham_3_term}
		\big|\langle j_g,f(\cdot,m_1^{\varepsilon},u_\varepsilon)\rangle_{m_1^{\varepsilon}}\big|\leq 6C_f V_g(m_1^{\varepsilon}).
	\end{equation} Similarly, 
	\begin{equation}\label{ineq:Ham_4_term}
		\big|\langle j_g,f(\cdot,m_2^{\varepsilon},u_\varepsilon)\rangle_{m_2^{\varepsilon}}\big|\leq 6C_f V_g(m_2^{\varepsilon}).
	\end{equation}
	Estimating the right-hand side of~\eqref{ineq:Ham_m_1,m_2} via inequalities~\eqref{ineq:Ham_fisrt_term}--\eqref{ineq:Ham_4_term}, we obtain that
	\begin{equation*} \begin{split}
			H(m_1^{\varepsilon}&,s_1^{\varepsilon})-H(m_2^{\varepsilon},s_2^{\varepsilon})\\ &\leq 
			2L\varepsilon^{-1}W_2^2(m_1^{{\varepsilon}},m_2^{{\varepsilon}})+6C_f\kappa\beta\big[V_g(m_1^{\varepsilon})]^\kappa +6C_f\kappa\beta\big[V_g(m_2^{\varepsilon})]^\kappa.
		\end{split} 
	\end{equation*}  From this,~\eqref{ineq:phi_1_phi_2} we obtain that 
	\begin{equation*} \Phi_\varepsilon(m_1^{\varepsilon},m_2^{\varepsilon})\leq (2L+2^{-1})\varepsilon^{-1}W_2^2(m_1^{\varepsilon},m_2^{\varepsilon}).\end{equation*}  Here, we additionally used the choice of $\kappa=1\wedge (6C_f)^{-1}$.  Hence, from~\eqref{ineq:phi_1_2_kappa} and the choice of \((m_1^\varepsilon,m_2^\varepsilon)\),
	\begin{equation*} \varkappa\leq {\Phi_\varepsilon}(\hat{m},\hat{m})\leq (2L+2^{-1})\varepsilon^{-1}W_2^2(m_1^{\varepsilon},m_2^{\varepsilon}).\end{equation*}  Taking into account the fact that $\varepsilon^{-1}W_2^2(m_1^{\varepsilon},m_2^{\varepsilon})\rightarrow 0$ as $\varepsilon\rightarrow 0$ (see~\eqref{convergence:dist_eps}), we conclude that 
	$\varkappa\leq 0$, which contradicts~\eqref{ineq:varkappa}.
\end{proof}

\section{Characterization of the value function}\label{sect:characterization}
We aim to show that the Kruzhkov transform of the value function~\eqref{intro:Kruzhkov} satisfies~\eqref{eq:hj},~\eqref{bound:hj} in the viscosity sense. In particular, this result ensures the existence of a viscosity solution for the system~\eqref{eq:hj},~\eqref{bound:hj}.

Our proof will use the relaxations of the time-optimal problem in the space of measures. To define them, for every $\varepsilon>0$, we recall the designation
\begin{equation*} M^\varepsilon\triangleq \operatorname{cl}\Big(\big\{m\in\mathcal{P}^2(\rd):\, \operatorname{dist}(m,M)\leq \varepsilon\big\}\Big).\end{equation*}  
Notice that
\begin{equation}\label{incl:M_varepsilon}
	M^{\varepsilon_1}\subset M^{\varepsilon_2}\text{ whenever }\varepsilon_1\leq\varepsilon_2
\end{equation} and
\begin{equation*} M=\bigcap_{\varepsilon>0}M^\varepsilon.\end{equation*}  

Now we  consider the time-optimal problem for the target set $M^\varepsilon$. As above, given $m(\cdot)\in C([0,+\infty);\mathcal{P}^2(\rd))$, we set 
\begin{equation*} \tau^\varepsilon(m(\cdot))\triangleq \inf\big\{t\in [0,+\infty):\, m(t)\in M^\varepsilon\big\}.\end{equation*}  By~\eqref{incl:M_varepsilon}, we have that, if $\varepsilon_1\leq\varepsilon_2$,
\begin{equation}\label{ineq:tau_eps}\tau^{\varepsilon_1}(m(\cdot))\geq \tau^{\varepsilon_2}(m(\cdot)).\end{equation} We define the value function for the time-optimal problem with the target set $M^\varepsilon$ by the rule:
\begin{equation*} \operatorname{Val}^\varepsilon(\mu)\triangleq \inf\Big\{\tau^\varepsilon(m(\cdot;\mu,\xi)):\,\xi\in\mathcal{U}\Big\}.\end{equation*}  It is convenient to consider the function $\operatorname{Val}^\varepsilon$ on $\operatorname{cl} G$. However, on $\operatorname{cl} G\cap M^\varepsilon$ it is obviously equal to $0$.  As for the original value function we have the following properties.
\begin{proposition}\label{prop:val_eps_semicontinuous} The function $\operatorname{Val}^\varepsilon$ is lower semicontinuous. Moreover, $\operatorname{Val}^\varepsilon(\mu)=0$ whenever $\mu\in M^\varepsilon$.     
\end{proposition}
\begin{proposition}\label{prop:optimal_existence_eps} If $\varepsilon>0$, $\mu\in G$ is such that $\operatorname{Val}^\varepsilon(\mu)<+\infty$. Then,  there exists $\xi_{\varepsilon}$ satisfying
	\begin{equation*} \operatorname{Val}^\varepsilon(\mu)=\tau^\varepsilon(m(\cdot;\mu,\xi_{\varepsilon})).\end{equation*} 
\end{proposition}
\begin{proposition}\label{prop:eps_dynprog}
	For every $\mu\in G$ and $h\in [0, \operatorname{Val}^\varepsilon(\mu)]$,
	$$  \operatorname{Val}^\varepsilon(\mu)-h=\inf_{\xi\in \mathcal{U}_{h}} \operatorname{Val}^\varepsilon(m(h;\mu,\xi)).$$
\end{proposition}
\begin{proposition}\label{prop:Val_eps_decrease}
	If $\varepsilon_1\leq \varepsilon_2$, then 
	\begin{equation*} \operatorname{Val}^{\varepsilon_1}(\mu)\geq \operatorname{Val}^{\varepsilon_2}(\mu).\end{equation*} 
\end{proposition}
\begin{proof} {Notice that, for each \(\xi\in\mathcal{U}\), 
		\begin{equation*} \tau^{\varepsilon_1}(m(\cdot,\mu,\xi))\geq \tau^{\varepsilon_2}(m(\cdot,\mu,\xi)).\end{equation*} 
		Taking the infimum on both sides with respect to the relaxed controls 
		\(\xi\), we obtain the statement of the proposition.}
\end{proof}

Recall that above we introduced the function~$\phi$ that is the Kruzhkov transform  of the value function $\operatorname{Val}$ (see~\eqref{intro:Kruzhkov}).

\begin{theorem}\label{th:character}
	The function $\phi$ is a viscosity solution to the Dirichlet problem~\eqref{eq:hj},~\eqref{bound:hj}.
\end{theorem}
\begin{proof} We split the proof into four steps.
	
	\textit{Step 1.} Here, we prove that the function $\phi$ is a supersolution to~\eqref{eq:hj}, \eqref{bound:hj}. First, notice that $\phi$ is bounded and lower semicontinuous, while $\phi(\mu)=0$ when $\mu\in\partial G$. 
	
	Furthermore, we choose a probability $\mu\in G$ and a function $s\in \partial^-\phi(\mu)$. By Remark~\ref{remark:sub_super_diff_eq}, we have that, for every $F\in \SL{\mu}$,
	\begin{equation}\label{ineq:sub_F}
		\left\langle s,F \right\rangle_\mu \leq d^-_H\phi(\mu;F).
	\end{equation}  The dynamic programming principle (see Theorem~\ref{th:dynprog}) can be rewritten as follows: for each $h\in [0,-\ln(1-\phi({\mu})]$,
	\begin{equation*} (1-\phi(\mu))e^h=\sup_{\xi\in\mathcal{U}_h}(1-\phi(m(h;\mu,\xi))).\end{equation*} 
	Due to Theorem~\ref{th:optima_existence},  there exists an optimal control $\xi_h\in\mathcal{U}_h$ such that 
	\begin{equation*} (1-\phi(\mu))e^h=(1-\phi(m(h;\mu,\xi_h))).\end{equation*}  Let 
	\begin{equation*} F_h(y)\triangleq \frac{1}{h}\int_0^h\int_U f(X(t;y,\mu,\xi_h),m(t;\mu,\xi_h),u)\xi_h(du|t)dt,\end{equation*}  where $X(\cdot;y,\mu,\xi_h)$ is a trajectory such that
	\begin{equation*} \frac{d}{dt}X(t;y,\mu,\xi_h)=\int_U f(X(t;y,\mu,\xi_h),m(t;\mu,\xi_h),u)\xi_h(du|t),\ \ X(0;y,\mu,\xi_h)=y.\end{equation*}  Due to  statement \ref{A:prop:continuity:3} of Proposition~\ref{A:prop:continuity} and  Hypothesis~\ref{hyp2}, we have that $F_h\in\SL{\mu}$. Notice that  $m(h;\mu,\xi_h)=(\operatorname{Id}+hF_h)\sharp\mu$. Thus, we have that
	\begin{equation}\label{equality:phi_F}\phi((\operatorname{Id}+hF_h)\sharp\mu)-\phi(\mu)=(\phi(\mu)-1)\cdot (e^h-1).\end{equation}
	Furthermore, Proposition~\ref{A:prop:F_limits} gives the existence of a sequence $\{h_k\}_{k=1}^\infty$ and a probability $\zeta\in\mathcal{P}(U)$ such that $h_k\rightarrow 0$, $F_{h_k}\rightarrow F$ as $k\rightarrow\infty$ for $F\in\SL{\mu}$ defined by the rule: $F(y)\triangleq \int_Uf(y,\mu,u)\zeta(du)$.  Therefore, from~\eqref{equality:phi_F}, we conclude that 
	\begin{equation*} d^-_H\phi(\mu;F)\leq \phi(\mu)-1.\end{equation*} 
	This and~\eqref{ineq:sub_F} imply that 
	\begin{equation*} \int_U\int_{\mathbb{R}^d} \left\langle s(x),f(x,\mu,u) \right\rangle \mu(dx)\zeta(du)\leq  \phi(\mu)-1. \end{equation*} 
	Using the definition of the Hamiltonian, we obtain the following inequality, for every $\mu\in G$ and $s\in\partial^-\phi(\mu)$,
	\begin{equation*} H(\mu,s)+1-\phi(\mu)\leq 0.\end{equation*} 
	Thus, $\phi$ is a supersolution of  problem~\eqref{eq:hj}, \eqref{bound:hj}.
	
	
	\textit{Step 2}. Let $\psi^\varepsilon$ be the Kruzhkov transform of the function $\operatorname{Val}^\varepsilon$, i.e., 
	\begin{equation*} \psi^\varepsilon(\mu) \triangleq 1-e^{-\operatorname{Val}^\varepsilon(\mu)}.\end{equation*}  Moreover, we define $\phi^\varepsilon$ to be an upper envelope of the function $\psi^\varepsilon$. This means that, for every $\mu\in {\operatorname{cl}}G$, 
	\begin{equation*} \phi^\varepsilon(\mu)\triangleq \limsup_{\mu'\rightarrow \mu,\ \ \mu'\in \operatorname{cl}G}\psi^\varepsilon(\mu')=\lim_{\delta\downarrow 0}\sup\big\{\psi^\varepsilon(\mu'):\, W_2(\mu,\mu')\leq\delta,\ \ \mu'\in {\operatorname{cl}G}\big\}.\end{equation*}  Notice that, if $\mu\in {\operatorname{cl}}G\cap M^\varepsilon$, then $\psi^\varepsilon(\mu)=0$. In particular,  $\phi^\varepsilon(\mu)=0$ whenever $\mu\in\partial G$.
	
	\textit{Step 3}. Here we prove that  each function $\phi^\varepsilon$ is a subsolution of~\eqref{eq:hj}, \eqref{bound:hj}. 
	
	The fact that  $\phi^\varepsilon$ satisfies the boundary condition was proved above. Moreover, the definition of $\phi^\varepsilon$ implies that it is upper semicontinuous and takes values in $[0, 1]$.
	
	Furthermore, let us show that the function $\phi^\varepsilon$ satisfies the following condition: for every $\mu\in G$ and $s\in \partial^+\phi^\varepsilon(\mu)$,
	\begin{equation}\label{ineq:H_phi_eps}
		H(\mu,s)+1-\phi^\varepsilon(\mu)\geq 0
	\end{equation} First, we consider the case where ${\phi^\varepsilon}(\mu)>0$. 
	
	By the construction of $\phi^\varepsilon$, there exists a sequence $\{\mu_k\}_{k=1}^\infty\subset {\operatorname{cl}G}$ converging to~$\mu$, such that
	\begin{equation*} \lim_{k\rightarrow\infty}\psi^\varepsilon(\mu_k)=\phi^\varepsilon(\mu).\end{equation*}  Without loss of generality, we assume that $\psi^\varepsilon(\mu_k)\geq \hat{h}\triangleq \phi^\varepsilon(\mu)/2$. Since ${\phi^\varepsilon}(\mu)>0$, we  have that $\hat{h}>0$.  By the definition of the superdifferential, for every $F\in \SL{\mu}$, we have
	\begin{equation*}
		\left\langle s,F \right\rangle_\mu \geq d^+_H\phi^\varepsilon(\mu;F).
	\end{equation*}  We choose $f^0(x)\triangleq f(x,\mu,u_0)$, where $u_0$ is such that \begin{equation}\label{intro:u_0} 
		H(\mu,s)=\left\langle s,f(\cdot,\mu,u_0)\right\rangle_\mu.\end{equation}
	Hence,
	\begin{equation}\label{ineq:H_derivative_geq}H(\mu,s)=  \left\langle s,f^0 \right\rangle_\mu\geq d^+_H\phi^\varepsilon(\mu;f^0). \end{equation} 
	Below, we evaluate $d^+_H\phi^\varepsilon(\mu;f^0)$ as follows.
	
	Since $\psi^\varepsilon$ is the Kruzhkov transform of the  function $\operatorname{Val}^\varepsilon$,  the dynamic programming principle for problem with the target $M^\varepsilon$ (see Proposition~\ref{prop:eps_dynprog}) implies that, for $h\in [0,-\ln(1-\psi^\varepsilon(\mu_k))]$,
	\begin{equation} \label{dynprog}
		(1-\psi^\varepsilon (\mu_k))e^h=\sup_{\xi\in \mathcal{U}_{h}}(1- \psi^\varepsilon(m(h;\mu_k,\xi))).    
	\end{equation}
	
	We consider  $h\in (0,\hat{h})$ and  the constant relaxed control on $\mathcal{U}_h$ defined by the rule $\xi^0(\cdot|t)\triangleq \delta_{u^0}$.     Inequality~\eqref{dynprog} gives that 
	\begin{equation*} (1-\psi^\varepsilon (\mu_k))e^h\geq 1- \psi^\varepsilon(m(h;\mu_k,\xi^0)).\end{equation*}  Using the definition of the function $\phi^\varepsilon$, the choice of the sequence $\{\mu_k\}_{k=1}^\infty$ and Proposition~\ref{lm0}, we conclude that 
	\begin{equation*} (1-\phi^\varepsilon (\mu))e^h\geq 1- \phi^\varepsilon(m(h;\mu,\xi^0)).\end{equation*} 
	Now, as above, let $X(\cdot;y,\mu,\xi^0)$ satisfy
	\begin{equation*} \frac{d}{dt}X(t;y,\mu,\xi^0)=f(X(t;y,\mu,\xi^0),m(t;\mu,\xi^0),u^0),\ \ X(0;y,\mu,\xi^0)=y.\end{equation*}  Set
	\begin{equation*} F_h(y)\triangleq \frac{1}{h}\int_0^h f(X(t;y,\mu,\xi^0),m(t),u^0)dt.\end{equation*}  Proposition~\ref{A:prop:F_limits} yields that 
	$F_h\rightarrow f^0$  as $h\rightarrow 0$ in $\SL{\mu}$. Additionally, notice that $X(h;y,\mu,\xi^0)=y+F_h(y)\cdot h.$ Thus,
	\begin{equation*} m(h;\mu,\xi^0)=(\operatorname{Id}+hF_h)\sharp \mu.\end{equation*}  
	Therefore,
	\begin{equation*} \phi^\varepsilon((\operatorname{Id}+h F_h)\sharp \mu)-\phi^\varepsilon(\mu)\geq (\phi^\varepsilon(\mu)-1)\cdot (e^h-1).\end{equation*}  Dividing both parts on $h$ and using the definition of the Hadamard upper derivative, we conclude that 
	\begin{equation*} d^+_H\phi^\varepsilon(\mu,f^0)\geq \phi^\varepsilon(\mu)-1.\end{equation*}  This and~\eqref{ineq:H_derivative_geq} imply~\eqref{ineq:H_phi_eps} for each $\mu\in G$ satisfying $\phi^\varepsilon(\mu)>0$ and every $s\in \partial^+\phi^\varepsilon(\mu)$.
	
	Now, let $\phi^\varepsilon(\mu)=0$. As above, we choose $s\in\partial^+\phi^\varepsilon(\mu)$. Since $\phi^\varepsilon$ is nonnegative on $\operatorname{cl} G$, we have that $d^+_H\phi^\varepsilon(\mu;f^0)\geq 0$. Here, $f^0(\cdot)=f(\cdot,\mu,u^0)$ with $u^0$ satisfying~\eqref{intro:u_0}. Thus (see~\eqref{ineq:H_derivative_geq}), 
	\begin{equation*} H(\mu,s)\geq d^+_H\phi^\varepsilon(\mu;f^0)\geq 0.\end{equation*}  This implies~\eqref{ineq:H_phi_eps} for each $\mu$ satisfying $\phi^\varepsilon(\mu)=0$ and every $s\in \partial^+\phi^\varepsilon(\mu)$.
	
	
	
	\textit{Step 4}. We claim that, for each $\mu\in G$,
	\begin{equation}\label{equality:lim_eps}\lim_{\varepsilon\downarrow 0}\phi^\varepsilon(\mu)=\phi(\mu). \end{equation} This equality directly follows from the inequality 
	\begin{equation}\label{ineq:limit_value}
		\lim_{\varepsilon\downarrow 0}\operatorname{Val}^\varepsilon(\mu)\geq \operatorname{Val}(\mu),
	\end{equation}  the monotonicity of the Kruzhkov transform and the comparison principle (see Theorem~\ref{th1}). Equality~\eqref{equality:lim_eps} and Steps 1--3 show that the function $\phi$ is a viscosity solution of~\eqref{eq:hj}, \eqref{bound:hj}.
	
	So, it remains to prove~\eqref{ineq:limit_value}.  First, recall that Proposition~\ref{prop:Val_eps_decrease} means that $\varepsilon\mapsto \operatorname{Val}^\varepsilon(\mu)$ decreases. Therefore, \(\lim_{\varepsilon\downarrow 0}\operatorname{Val}^\varepsilon(\mu)\) exists.

	Now, we consider the case where 
	\begin{equation}\label{converge:Val_eps}
		T\triangleq\lim_{\varepsilon\downarrow 0}\operatorname{Val}^\varepsilon(\mu)<\infty.
	\end{equation}     
	
	By Proposition \ref{prop:optimal_existence_eps}, for each $\varepsilon>0$, there exists a relaxed control $\xi_{\varepsilon}$ such that
	\begin{equation}\label{equality:Val_xi_eps}\operatorname{Val}^\varepsilon(\mu)=\tau^\varepsilon(m(\cdot;\mu,\xi_{\varepsilon}))\end{equation} and  $m(\operatorname{Val}^\varepsilon(\mu);\mu,\xi_{\varepsilon})\in M^\varepsilon$. Furthermore, the fact that $\mathcal{U}$ is compact  implies that there exist a sequence $\{\varepsilon_k\}_{k=1}^\infty$ converging to zero and a relaxed control $\xi$ such that $\xi_{\varepsilon_k}$ tend to $\xi$. Thus, by Proposition~\ref{lm0}, we have that 
	$\{m(\cdot;\mu,\xi_{\varepsilon_k})\}_{k=1}^\infty$ converges to $m(\cdot;\mu,\xi)$. Using the definition of the sets \(M^\varepsilon\), together with  equality~\eqref{equality:Val_xi_eps} and  convergence~\eqref{converge:Val_eps}, we conclude that
	\(m(T;\mu,\xi)\in M\). Hence, \begin{equation*} \operatorname{Val}(\mu)\leq T=\lim_{\varepsilon\downarrow 0}\operatorname{Val}^\varepsilon(\mu).\end{equation*} 
	
	The case where 
	\begin{equation*} \lim_{\varepsilon\downarrow 0}\operatorname{Val}^\varepsilon(\mu)=\infty\end{equation*}  is obvious.
\end{proof}

\section{Topological properties of the value function}\label{sect:topological}

First, we provide a sufficient condition for the continuity of the value function.
\begin{theorem} \label{th3}
	Assume that the viscosity solution of the Dirichlet problem~\eqref{eq:hj},~\eqref{bound:hj} is continuous {at each point of $\partial G$}. Then, it is continuous on $\operatorname{cl}G$.
\end{theorem}
\begin{remark}
	Theorem~\ref{th3} implies that, if the value function is continuous {at every point of} $\partial G$, then it is continuous on $\operatorname{cl}G$. Moreover, the continuity of the value function {at every point of} $\partial G$ follows directly from the small-time local attainability property introduced in \cite{Krastanov_Quincampoix_2001}.
	
	In the case considered here, this property takes the following form: for every $\varepsilon>0$, there exists $\delta>0$ such that $\operatorname{Val}(\mu)\leq \varepsilon$ whenever $\mu\in G$ with $\operatorname{dist}(\mu,M)<\delta$. If, in addition, $M$ is compact, then the small-time local attainability property is equivalent to the continuity of the value function {at every point of} $\partial G$. In the space of probability measures, this property was studied in \cite{Cavagnari2017,Cavagnari2021}. In those works, the regularity of the value function for the time-optimal problem in the space of probability measures was investigated under the assumption that the problem satisfies the small-time local attainability property, with the target being a hyperplane in the space of measures.
\end{remark}
\begin{proof}[Proof of Theorem~\ref{th3}]
	Let $\phi$ denote the aforementioned viscosity solution.  For every $\mu\in\operatorname{cl}G$, we put
	\begin{equation*} \phi^*(\mu)\triangleq \limsup_{\mu'\rightarrow \mu,\, {\mu'\in\operatorname{cl}G}}\phi(\mu').\end{equation*}  Notice that the function $\phi^*:\operatorname{cl}G\rightarrow \mathbb{R}$ is upper semicontinuous {and takes values in \([0,1]\)}. 
	
	{The function \(\phi^*\) is a subsolution of~\eqref{eq:hj}. 
		The proof of this property follows literally from the arguments in Step~3 of Theorem~\ref{th:character}, with \(\psi^\varepsilon\) and \(\phi^\varepsilon\) replaced by \(\phi\) and \(\phi^*\), respectively.
		As in that proof, the key property is the dynamic programming principle (see Theorem~\ref{th:dynprog}), which, due to the Kruzhkov transform, takes the form
		\begin{equation*} 
			(1-\phi(\mu))e^h=\sup_{\xi\in\mathcal{U}_h}(1-\phi(m(h;\mu,\xi))).
	\end{equation*} }
	
	Furthermore, the continuity of the function $\phi$ {at every point of} $\partial G$ gives that $\phi^*(\mu)=0$ for every $\mu\in\partial G$. Thus,
	$\phi^*$ is a subsolution of~\eqref{eq:hj}, \eqref{bound:hj}. The comparison principle (see Theorem~\ref{th1}) gives that 
	$\phi^*(\mu)\leq \phi(\mu)$ on $\operatorname{cl} G$. Since the opposite inequality directly follows from the definition of the function $\phi^*$, we conclude that $\phi=\phi^*$ and the function $\phi$ is continuous.
\end{proof}

Now, let us consider the perturbation of the dynamics of the time-optimal problem. Since the value function is generally only lower semicontinuous, we require a suitable concept of convergence. In our opinion, the most natural convergence here is the $\Gamma$-convergence defined as follows (see \cite[Definition 1.5]{Braides}).
\begin{definition}\label{gc}
	For each natural $n$, let $F_n$ be a  functional  defined on a topological space $X$ with values in $\mathbb{R}\cup \{+\infty\}$. The sequence $\{F_n\}_{n=1}^\infty$ is said to $\Gamma$-converge  to  a functional $F:X\rightarrow \mathbb{R}\cup \{+\infty\}$ if the following two conditions hold:
	\begin{enumerate}
		\item for every sequence $\{x_n\}_{n=1}^\infty \in X$ satisfying $\lim\limits_{n\rightarrow \infty}x_n=x$, \begin{equation*} F(x)\leq \liminf\limits_{n\rightarrow \infty} F_n(x_n);\end{equation*} 
		\item for each $x\in X$, there is a sequence  $\{x_n\}_{n=1}^\infty \in X$ such that \begin{equation*} \lim\limits_{n\rightarrow \infty}x_n=x \text{ and }F(x)\geq \limsup\limits_{n\rightarrow \infty} F_n(x_n).\end{equation*} 
	\end{enumerate}
\end{definition} If $\{F_n\}_{n=1}^\infty$ $\Gamma$-converges to $F$, then we write
\begin{equation*} F=\Gamma\mbox{-}\lim\limits_{n\rightarrow\infty} F_n.\end{equation*} 

Now we  consider a perturbation of the dynamics of~\eqref{dyn}. For each natural $n$, let a function $f_n:\rd\times\mathcal{P}^2(\rd)\times U\rightarrow \rd$, $n\in\mathbb{N}$ satisfy Hypothesis~\ref{hyp1}--\ref{hyp3}.  If  $\mu\in\mathcal{P}^2(\rd)$, $\xi\in\mathcal{U}$, then there exists a unique distributional solution to
\begin{equation*} \partial_t m(t)+\operatorname{div}\Bigg(\int_U f_n(x,m(t),u)m(t)\xi(du|t)\Bigg)=0,\ \ m(0)=\mu. \end{equation*}  We denote this solution by $m_n(\cdot;\mu,\xi)$. The value function of the perturbed time-optimal is defined by the rule:
\begin{equation*} \operatorname{Val}_n(\mu)\triangleq \inf\big\{\tau(m_n(\cdot;\mu,\xi):\, \xi\in\mathcal{U}\big\}.\end{equation*}

\begin{theorem} {Let the functions \(f_n\) and \(f\) be Lipschitz continuous in \(x\) and \(m\) with a common
		Lipschitz constant \(L\).} In addition, assume the following convergence of
	the dynamics \(f_n\) to \(f\): assume the following convergence of the dynamics $f_n$ to $f$: for each $c>0$, 
	\begin{equation*}
		\begin{split}
			\sup\Big\{\|f_n(x,&m,u)-f(x,m,u)\|:\\ &x\in\rd,\, m\in\mathcal{P}^2(\rd),\, \varsigma(m)\leq c,\, u\in U\Big\}\rightarrow 0\text{ as }n\rightarrow\infty.
		\end{split}
	\end{equation*} Then,  
	$$\Gamma\mbox{-}\lim\limits_{n\rightarrow\infty} \operatorname{Val}_n=\operatorname{Val} .$$
\end{theorem}
\begin{proof} 
	
	Let us prove  the  first property in Definition~\ref{gc}. Let $\{\mu_n\}_{n=1}^\infty\in \mathcal{P}^2(\mathbb{R}^d)$ be a sequence converging to $\mu$. If ${\liminf_{n\rightarrow\infty}}\operatorname{Val}_n(\mu_n)= +\infty$, then, obviously,
	\begin{equation*} \operatorname{Val}(\mu)\leq \liminf_{n\rightarrow\infty}\operatorname{Val}_n(\mu_n).\end{equation*}  If the {lower} limit is finite,  without loss of generality, we can assume that 
	the whole sequence $\{\operatorname{Val}_n(\mu_n)\}_{n=1}^\infty$ converges to some number $\theta$.
	By Theorem~\ref{th:optima_existence}, for each natural $n$, there exists a control $\xi_n\in\mathcal{U}$ such that 
	$$\theta_n\triangleq\tau(m_n(\cdot;\mu_n,\xi_n))=\operatorname{Val}_n(\mu_n).$$ Additionally, for sufficiently large $n$ one has that $\tau(m_n(\cdot;\mu_n,\xi_n))<{\theta+1}$. Since $\mathcal{U}$ is compact, passing, if
	necessary, to a subsequence, we have that 
	$\xi_n\rightarrow \xi$ as $n\rightarrow\infty$. Thus,  Proposition~\ref{lm0} gives that $\{m_n(\cdot;\mu_n,\xi_n)\}_{n=1}^\infty$ converges to $m(\cdot;\mu,\xi)$ on $C([0,{\theta+1}],\mathcal{P}^2(\rd))$. {Moreover, Statement~\ref{A:prop:continuity:2} of Proposition~\ref{A:prop:continuity} gives that all functions \(t\mapsto m(\cdot;\mu,\xi)\), \(t\mapsto m_n(\cdot;\mu_n,\xi_n)\) are uniformly continuous on \([0,\theta+1]\).} Since,
	$m_n({\theta_n};\mu_n,\xi_n)\in M,$ whereas $M$ is closed, we have that $m(\theta;\mu,\xi)\in M$. Hence,
	\begin{equation*} \operatorname{Val}(\mu)\leq \tau(m(\cdot;\mu,\xi))\leq\theta=\liminf_{n\rightarrow\infty}\operatorname{Val}_n(\mu_n).\end{equation*} 
	This proves the first property of the $\Gamma$-convergence. 
	
	To prove the second property, we choose a measure $\mu\in {\operatorname{cl} G}$. First, we assume that ${0<}\operatorname{Val}(\mu)<\infty$. {This, in particular, gives that \(\mu\in G\).} As above, by Theorem~\ref{th:optima_existence}, one can find a control $\xi\in\mathcal{U}$, such that  $\operatorname{Val}(\mu)=\tau(m(\cdot;\mu,\xi))$. For shortness, we put $\theta\triangleq \operatorname{Val}(\mu)$, $\nu\triangleq m(\theta,\mu,\xi)$. Notice that $\nu\in M$. Now, let $m'_n(\cdot):[0,\theta]\rightarrow \mathcal{P}^2(\rd)$ satisfy
	\begin{equation*} \partial_t m_n'(t)+\operatorname{div}\Bigg(\int_U f_n(\cdot,m_n'(t),u)\xi(du|t)m_n'(t)\Bigg)=0,\ \ m_n'(\theta)=\nu.\end{equation*}  One can use time reversal to prove the existence of such a motion. Additionally, Proposition~\ref{lm0} applied in the reverse time gives that  \begin{equation*} \sup_{t\in [0,\theta]}W_2(m_n'(t),m(t,\mu,\xi))\rightarrow 0\text{ as }n\rightarrow \infty.\end{equation*}  We put $\mu_n\triangleq m_n'(0)$. By construction, $m_n(\cdot,\mu_n,\xi)=m_n'(\cdot)$. Moreover, $W_2(\mu_n,\mu)$ tends to $0$ when $n\rightarrow\infty$. In particular, $\mu_n\in G$ for sufficiently large $n$. Thus, due to the fact that $m_n'(\theta)=\nu\in M$, one has that $\tau(m_n(\cdot,\mu_n,\xi))\leq \theta.$ Hence, $$\operatorname{Val}_n(\mu_n)\leq \tau(m_n(\cdot;\mu_n,\xi))\leq \tau(m(\cdot;\mu,\xi))=\operatorname{Val}(\mu).$$ This means that  $\limsup_{n\rightarrow\infty}\operatorname{Val}_n(\mu_n)\leq\operatorname{Val}(\mu). $

	If $\operatorname{Val}(\mu)=+\infty$, then, for an arbitrary sequence $\{\mu_n\}_{n=1}^\infty\in G$ converging to  $\mu\in G$, one has that $\limsup_{n\rightarrow\infty}\operatorname{Val}_n(\mu_n)\leq\operatorname{Val}(\mu). $ {If \(\operatorname{Val}(\mu)=0\) (i.e., \(\mu\in M\)), it suffices to set \(\mu_n\triangleq\mu\).}
\end{proof}

	
	\section{An example}\label{sect:example}
	We consider a time-optimal problem for the nonlocal continuity equation assuming that $d=1$, $U=[-1,0]$,
	\begin{equation*} f(x,m,u)=u-\int_{\mathbb{R}}ym(dy),\end{equation*}  while
	the target set is $$M=\Bigg\{ m\in \mathcal{P}^2(\mathbb{R}): \  \int_{\mathbb{R}} ym(dy)=0 \Bigg\}.$$
	Notice that this system satisfies Hypotheses~\ref{hyp1}--\ref{hyp3}.
	
	In the examined case, the nonlocal continuity equation~\eqref{dyn} takes the form:
	$$ \partial_t m(t) +\mbox{div}\Bigg(\bigg(u(t)-\int_{\mathbb{R}} ym(t,dy)\bigg)m(t) \Bigg)=0, \ m(0)=\mu\in \mathcal{P}^2(\mathbb{R}).$$
	
	The mean of the measure $m(t)$ defined by the rule: \begin{equation*} \bar{m}(t)\triangleq \int_{\mathbb{R}}ym(t,dy)\end{equation*} 
	obeys the following equation:
	\begin{equation*}\label{eq:E} \frac{d}{dt}\bar{m}(t)= u-\bar{m}(t),\ \ \bar{m}(0)=\bar{\mu}.\end{equation*} Here, given a measure $\mu$, we  put  $\bar{\mu}\triangleq \int_{\mathbb{R}}y\mu(dy)$. 
	Notice that the target set also is determined by the mean, i.e, the aim of the control is to find the first time $\tau$ such that  $\bar{m}(\tau)= 0$.
	Given an open-loop control $u(\cdot)$, we have that the corresponding evolution of the mean is $$ \bar{m}(t)=\Bigg(\bar{\mu}+\int_0^tu(t')e^{t'}dt'\Bigg)e^{-t}.$$ Thus, the optimal control is as follows:
	\begin{enumerate}
		\item  $u(\cdot)\equiv -1$ if  $\bar{\mu}\geq0$; 
		\item no optimal control if  $\bar{\mu}<0$;
	\end{enumerate} whereas  the value function is
	\begin{equation*} 
		\operatorname{Val}(\mu)=\left\{ \begin{array}{ll}
			\ln(\bar{\mu}+1), & \bar{\mu}\geq 0, \\
			+\infty, & \text{otherwise.}
		\end{array}
		\right.
	\end{equation*} 
	
	Notice that the value function is only lower semicontinuous.
	
	Now, let us write down the  Hamiltonian corresponding to the examined problem. It has the form:
	\begin{equation*}
		\begin{split}
			H(m,s)&=\inf\limits_{u\in[-1;0]}
			\int_{\mathbb{R}} s(x) \Bigl(u-\int_{\mathbb{R}} ym(dy)\Bigr) m(dx)  \\
			&= 
			\begin{cases}
				\Bigl(-1 -\bar{m}\Bigr)\int_{\mathbb{R}} s(x)  m(dx) , &  \int_{\mathbb{R}} s(x)  m(dx)>0,\\
				-\bar{m}\int_{\mathbb{R}} s(x)  m(dx), & \int_{\mathbb{R}} s(x)  m(dx)\leq 0 .
			\end{cases}
		\end{split}
	\end{equation*} Thus, the Dirichlet problem is
	\begin{equation*} H(m,\nabla\phi(m))+1-\phi(m)=0,\ \ \int_{\mathbb{R}}ym(dy)\neq 0;\ \ \phi(m)=0\text{ when } \int_{\mathbb{R}}ym(dy)= 0.\end{equation*} 
	
	Due to  the Kruzhkov transform of the function $\operatorname{Val}$,
	the viscosity solution is equal to
	$$ \phi(\mu)=\begin{cases}
		1-e^{-\ln(1+\bar{\mu})}, &  \bar{\mu}\geq0,\\
		1, & \bar{\mu}<0.
	\end{cases}$$
	
	\appendix
	
	\section{{Space dual of \(\SL{m}\)}}\label{sect:app_dual}
	
	{
		The aim of this section is to show that the dual space of \(\SL{m}\) is isomorphic to \(\ID{m}\). To this end, given \(m\in\mathcal{P}^2(\rd)\), we denote by \(\SC{m}\) the space \(L^\infty(\rd,m;\rd)\) with the following convergence structure: a sequence \(\{G_n\}_{n=1}^\infty\subset\SC{m}\) converges to \(G\in\SC{m}\) provided there exists a constant \(C>0\) such that \(\|G_n(x)\|\leq C\) for all \(n\) and for \(m\)-a.e. \(x\in\rd\), and \(G_n(x)\to G(x)\) \(m\)-a.e. The convergence on \(\SC{m}\) is weaker than the one generated by the standard \(L^\infty\)-norm.}
	
	{	Notice that \(\SL{m}\) and \(\SC{m}\) are isomorphic as sequential convergence vector spaces via the linear mapping \(\mathcal{A}:\SL{m}\to\SC{m}\) defined by
		\(
		(\mathcal{A}[F])(x)\triangleq F(x)/\mathscr{a}(x)
		\).}
	
	{ \begin{proposition}\label{dual:prop:dual}
			The dual space of \(\SC{m}\) is isomorphic to \(L^1(\rd,m;\rd)\); that is,
			if \(r\in L^1(\rd,m;\rd)\) corresponds to a continuous linear functional
			\(\Psi:\SC{m}\rightarrow\mathbb{R}\), then, for each \(G\in\SC{m}\),
			\(
			\Psi[G]=\langle r,G\rangle_m\).
		\end{proposition}\begin{proof}
			In what follows, let \(\mathcal{B}(\rd)\) denote the Borel \(\sigma\)-algebra on \(\rd\).
			
			Let \(\Psi:\SC{m}\rightarrow\mathbb{R}\) be a continuous linear functional. Since the convergence generated by the norm \(\|\cdot\|_{L^\infty(\rd,m;\rd)}\) is stronger than the one introduced on \(\SC{m}\), \(\Psi\) is also a continuous linear functional from \(L^\infty(\rd,m;\rd)\) to \(\mathbb{R}\). Then, by~\cite[Theorem I.1.13]{diestel1977vector}, there exists a finitely additive vector measure
			\(\chi:\mathcal{B}(\rd)\to\rd\) such that \(\chi\) vanishes on \(m\)-null sets and, for each \(G\in\SC{m}\),
			\begin{equation}\label{dual:equality:Psi}
				\Psi[G]=\int_{\rd}\langle G(x),\chi(dx)\rangle.
			\end{equation}
			
			We claim that \(\chi\) is \(\sigma\)-additive. Otherwise, there exist a vector \(\xi\in\rd\) and a sequence of sets \(\{E_n\}_{n=1}^\infty\subset \mathcal{B}(\rd)\) such that
			\begin{equation}\label{A:choice:E} 
				E_{n+1}\subset E_n,\qquad \bigcap_{n=1}^\infty E_n=\varnothing,\qquad
				\lim_{n\rightarrow \infty}\langle\xi,\chi(E_n)\rangle >0.
			\end{equation} 
			We let
			\begin{equation*} 
				G_n(x)\triangleq \xi\mathbbm{1}_{E_n}(x),
			\end{equation*} 
			where \(\mathbbm{1}_E\) is the indicator function for \(E\subset \rd\).

			We have \(G_n\to \mathbf{0}\) in \(\SC{m}\) as \(n\to\infty\). Here, as above, \(\mathbf{0}\) is the function assigning the zero vector to each \(x\in\rd\). Indeed, \(\|G_n(x)\|\leq \|\xi\|\) and \(\|G_n(x)\|\to 0\) for \(m\)-a.e. \(x\in\rd\). At the same time (see~\eqref{dual:equality:Psi} and~\eqref{A:choice:E}),
			\begin{equation*} 
				\lim_{n\rightarrow \infty}\Psi[G_n]
				=\lim_{n\rightarrow \infty}\langle \xi,\chi(E_n)\rangle>0.
			\end{equation*} 
			This contradicts the continuity of the functional \(\Psi\).
			
			Finally, due to the Radon--Nikodym theorem~\cite[Theorem III.3.1]{diestel1977vector}, there exists a function \(r\in L^1(\rd,m;\rd)\) such that
			\begin{equation*} 
				\Psi[G]=\int_{\rd}\langle G(x),\chi(dx)\rangle
				=\int_{\rd}\langle r(x),G(x)\rangle\, m(dx).
			\end{equation*} 
			
			Conversely, for each \(r\in L^1(\rd,m;\rd)\), the mapping
			\[
			\SC{m}\ni G\mapsto \int_{\rd}\langle r(x), G(x)\rangle\,m(dx)
			\]
			is continuous; this follows from Lebesgue's dominated convergence theorem.
	\end{proof}}
	
	{ 
		\begin{corollary}\label{dual:corollary:SL}
			The dual space of \(\SL{m}\) is isomorphic to \(\ID{m}\); if \(s\in\ID{m}\) corresponds to a continuous linear functional \(\Phi:\SL{m}\to\mathbb{R}\), then \(\Phi[F]=\langle s,F\rangle_m\) for every \(F\in\SL{m}\).
		\end{corollary}
		\begin{proof} If \(\Phi:\SL{m}\rightarrow\mathbb{R}\) is a continuous linear functional, then we define a continuous linear functional \(\Psi:\SC{m}\rightarrow\mathbb{R}\) by letting
			\(\Psi(G)\triangleq \Phi(\mathcal{A}^{-1}G)\). Using Proposition~\ref{dual:prop:dual}, we have that there exists \(r\in L^1(\rd,m;\rd)\) such that
			\(\Psi(G)=\langle r,G\rangle_m\). Thus, for \(F\in \SL{m}\), we have that
			\(\Phi(F)=\Psi(\mathcal{A}F)=\langle r\mathscr{a}^{-1},F\rangle_m\). Setting \(s\triangleq r\mathscr{a}^{-1}\), we conclude that \(s\in \ID{m}\) and \(\Phi(F)=\langle s,F\rangle_m\).
			
			The converse property, namely that
			\(\SL{m}\ni F\mapsto \langle s,F\rangle_m\)
			is a continuous functional, follows from Lebesgue's dominated convergence theorem.
	\end{proof}}
	
	\section{Properties of relaxed controls}\label{sect:A}
	
	For given  $T>0$, $\mu\in\mathcal{P}^2(\rd)$, $\xi\in \mathcal{U}_T$ and $m(\cdot)\triangleq m(\cdot;\mu,\xi)$, we denote by $X(\cdot;y,\mu,\xi)$ a solution of the following initial value problem:
	\begin{equation}\label{A:intro:X}
		\frac{d}{dt}x(t)=\int_U f(x(t),m(t),u)\xi(du|t),\ \ x(0)=y.
	\end{equation}

	\begin{proposition}\label{A:prop:continuity} The following estimates hold for every \(t,s\in[0,T]\) and \(y\in\rd\), with constants \(c_0,c_1,c_2,c_3>0\) depending only on the Lipschitz constant of \(f\), on \(T\), and on \(R\) (where \(\varsigma(\mu)\leq R\)):
		\begin{enumerate}
			\item\label{A:prop:continuity:1} \(\varsigma(m(t))\leq c_0\);
			\item\label{A:prop:continuity:2} \(W_2(m(t),m(s))\leq c_1|t-s|\);
			\item\label{A:prop:continuity:3} \(\|X(t;y,\mu,\xi)\|\leq c_2(1+\|y\|)\);
			\item\label{A:prop:continuity:4} \(\|X(t;y,\mu,\xi)-y\|\leq c_3(1+\|y\|)t\).
		\end{enumerate}
	\end{proposition}
	\begin{proof}
		The proof relies on the integral representation of~\eqref{A:intro:X} 
		\begin{equation}\label{A:equality:X_integral}X(t;y,\mu,\xi)=y+\int_0^t \int_U f(X(t';y,\mu,\xi),m(t'),u)\xi(du|t')dt'\end{equation} and the upper bound on $f$ that directly follows from Hypothesis~\ref{hyp2}
		\begin{equation}\label{A:ineq:f}
			\|f(x,m,u)\|\leq C'(1+\|x\|+\varsigma(m)).
		\end{equation} Here, \(C'\) is a constant depending only on the Lipschitz constant for the function \(f\). 
		
		Thus, taking into account the fact that $\varsigma(m(t))=\|X(t;\cdot,\mu,\xi)\|_{L^2(\rd,\mu;\rd)}$ and the Minkowski inequalities, we conclude that 
		\begin{equation*} \varsigma(m(t))\leq \varsigma(\mu)+C'\int_0^t(1+2\varsigma(m(t'))dt'.\end{equation*}  This and the Gronwall's inequality imply the first statement of the proposition.
		
		{To show the second statement, it suffices to notice that
			\begin{equation*} W_2(m(t),m(s))\leq \|X(t;\cdot,\mu,\xi)-X(s;\cdot,\mu,\xi)\|_{L^2(\rd,\mu;\rd)},
			\end{equation*} 
			use the equality
			\begin{equation*} 
				X(t;y,\mu,\xi)=X(s;y,\mu,\xi)+\int_s^t \int_U f(X(t';y,\mu,\xi),m(t'),u)\xi(du|t')\,dt',
			\end{equation*} 
			and apply estimate~\eqref{A:ineq:f} together with the first statement. We assume, without loss of generality, that \(s<t\).}
		
		The third and the fourth statements directly follow from integral representation of  solution~\eqref{A:equality:X_integral}, estimate~\eqref{A:ineq:f}, the first statement of this proposition and the Gronwall's inequality.
	\end{proof}
	\begin{remark} One can use Hypothesis~\ref{hyp3} and prove that 
		\begin{equation*} \int_{\rd}\mathscr{a}(x)m(t,dx)\leq c_0'\end{equation*}  for some constant $c_0'$. Here, $\mathscr{a}(x) = \sqrt{1 + \|x\|^2}$ (as defined in~\eqref{intro:almost_norm}).
	\end{remark}

	\begin{proposition}\label{A:prop:F_limits} Let the following be given:
		\begin{itemize}
			\item \(\mu\in\mathcal{P}^2(\rd)\);
			\item \(T>0\);
			\item for each \(h\in(0,T]\), a control \(\xi_h\in\mathcal{U}_h\);
			\item the function
			\begin{equation*} 
				F_h(y)\triangleq h^{-1}\int_{[0,h]\times U}
				f(X(t;y,\mu,\xi_h),m(t;\mu,\xi_h),u)\,\xi_h(d(t,u)).
			\end{equation*} 
		\end{itemize}
		Then there exist a sequence \(\{h_k\}_{k=1}^\infty\subset(0,T]\) with \(h_k\to0\) and a probability \(\zeta\in\mathcal{P}(U)\) such that
		\(
		F_{h_k}\to F \quad\text{in }\SL{\mu},
		\)
		where \(F\in\SL{\mu}\) is defined by
		\(
		F(y)\triangleq \int_U f(y,\mu,u)\zeta(du).
		\)
		If, in addition, there exists a probability \(\zeta_0\in\mathcal{P}(U)\) such that
		\(
		\xi_h(\cdot|t)=\zeta_0
		\)
		for every \(h\in(0,T]\) and \(t\in[0,h]\) (or for a.e. \(t\), depending on your convention), then, provided the representation is unique, \(\zeta=\zeta_0\) and
		\begin{equation*} 
			F_h\to F \quad\text{in }\SL{\mu}\quad\text{as }h\to0^+.
		\end{equation*}   
	\end{proposition}
	\begin{proof}
		First, notice that due to statements \ref{A:prop:continuity:1}, \ref{A:prop:continuity:3} of Proposition~\ref{A:prop:continuity} and  inequality~\eqref{A:ineq:f}, there exists a constant $C$ such that, for each $h\in [0,T]$ and $y\in\rd$, 
		\begin{equation}\label{A:ineq:F_h_C}\|F_h(y)\|\leq C\mathscr{a}(y).\end{equation} 
		Furthermore, we define the measure $\zeta_h\in\mathcal{P}(U)$ by the rule: for every Borel set  $\Gamma\subset U$, $$\zeta_h(\Gamma)\triangleq h^{-1}\xi_h([0,h]\times\Gamma).$$ 
		Since $U$ is compact, due to the Prokhorov's theorem there exists a sequence $\{h_k\}_{k=1}^\infty$ converging to zero such that $\{\zeta_{h_k}\}_{k=1}^\infty$ narrowly converges to some $\zeta\in \mathcal{P}(U)$. Due to   Hypothesis~\ref{hyp2}, given \(y\in\rd\),
		\begin{equation*} \begin{split}
				\|F_{h_k}(y)-F(y)\|\leq h_k^{-1}L\int_0^{h_k}\big(&\|X(t,y,\mu,\xi_{h_h})-y\|dt+ W_2(m(t,\mu,\xi_{h_k}),\mu)\big)dt\\&+\Bigg\|\int_{U}f(y,\mu,u)\zeta_{h_k}(du)-\int_{U}f(y,\mu,u)\zeta_0(du)\Bigg\|.\end{split}\end{equation*} 
		The first term tends to zero thanks to Proposition~\ref{A:prop:continuity}. Since $\zeta_{h_k}$ converges to $\zeta$ narrowly, we have that the second term also vanishes for each \(y\in\rd\). This together with~\eqref{A:ineq:F_h_C} implies the convergence of $\{F_{h_k}\}_{k=1}^\infty$ to $F$ in $\SL{\mu}$.
		
		In the case where $\xi_h(\cdot|t)=\zeta_0$, we, as above, by Hypothesis~\ref{hyp2},  have that 
		\begin{equation*} \|F_h(y)-F(y)\|\leq h^{-1}L\int_0^h\big(\|X(t,y,\mu,\xi_h)-y\|+W_2(m(t,\mu,\xi_h),\mu)\big)dt.\end{equation*}  The desired convergence now follows directly from this and estimate~\eqref{A:ineq:F_h_C}.
	\end{proof}
	
	\section{Moment gauge function}\label{sect:B}	
	The  aim of this section is to prove Lemmas~\ref{B:lm:g},~\ref{B:lm:derivative_g}.
	
	\begin{proof}[Proof of Lemma~\ref{B:lm:g}]
		First, notice that{, since $m\in\mathcal{P}^2(\rd)$,} 
		\begin{equation*} \lim_{R\rightarrow \infty}\int_{\mathscr{a}(x)\geq R}\mathscr{a}^2(x)m(dx)\rightarrow 0.\end{equation*}  
		
		Now we construct an increasing sequence $\{R_n\}_{n=0}^\infty$ inductively. 
		First, we put $R_0\triangleq 0$, $R_1\triangleq 1$.
		If we already constructed numbers $R_1<\ldots< R_{n}$, then we choose $R_{n+1}\geq (R_{n}+1) $  
		such that
		\begin{equation}\label{B:ineq:R_n}
			\int_{\mathscr{a}(x)\geq R_{n+1}}{\mathscr{a}^2}(x)m(dx)\leq (n+1)^{-1} 2^{-n}
		\end{equation} and
		\begin{equation*} nR_{n+1}\geq (n+2) R_n.\end{equation*}  Notice that the latter inequality can be rewritten in the equivalent form
		\begin{equation}\label{B:ineq:R_n_g_0}
			\frac{1}{R_{n+1}-R_n}\leq \frac{n+2}{2R_{n+1}}.
		\end{equation}  Now we define $g_0(r)$ by the rule:
		\begin{itemize}
			\item if $r\in [R_0,R_1]$, then ${g_0}(r)\triangleq 3$;
			\item if $r\in [R_n,R_{n+1}]$ for some $n\geq 1$, then we put 
			\begin{equation*} g_0(r)\triangleq (n+2) +(r-R_n)(R_{n+1}-R_n)^{-1}.\end{equation*}  
		\end{itemize}

		Notice that the function $g_0$ is positive, monotonously increasing on $[1,+\infty)$ and divergent. 
		Since $g_0(r)\leq n+3$ when $r\in [R_{n},R_{n+1}]$, we have that
		\begin{equation}\label{B:ineq:int}
			\begin{split}
				\int_{\rd}&\mathscr{a}^2(x)g_0(\mathscr{a}(x))m(dx)\\&=\sum_{n=1}^\infty\int_{R_{n}\leq \mathscr{a}(x)<R_{n+1}}\mathscr{a}^2(x)g_0(\mathscr{a}(x))m(dx) 
				\\&\leq \sum_{n=1}^\infty\int_{R_{n}\leq \mathscr{a}(x)<R_{n+1}}\mathscr{a}^2(x)(n+3)m(dx)\\&\leq 4\int_{\rd} \mathscr{a}^2(x)m(dx)+\sum_{n=2}^\infty \int_{R_{n}\leq \mathscr{a}(x)}\mathscr{a}^2(x)n m(dx).
			\end{split}
		\end{equation} From this, the fact that $\int_{\rd} \mathscr{a}^2(x)m(dx)<\infty$ and \eqref{B:ineq:R_n}, we conclude that $\int_{\rd}\mathscr{a}^2(x)g_0(\mathscr{a}(x))m(dx)$ is finite.
		
		The function $g_0$ is differentiable on each interval $(R_n,R_{n+1})$. The derivative on each interval $(R_n,R_{n+1})$ is constant and equal to $A_n$, where $A_0=0$ and  $A_n\triangleq (R_{n+1}-R_n)^{-1}$ for $n\geq 1$. Notice that $A_n\in [0,1]$. From~\eqref{B:ineq:R_n_g_0}, we have that, if $r\in [R_n,R_{n+1}]$, $n\geq 0$, then 
		\begin{equation}\label{B:ineq:A}A_n\leq \frac{n+2}{2R_{n+1}}\leq \frac{g_0(r)}{2r}.\end{equation} 
		
		Unfortunately, the function $g_0$ is only piecewise smooth. Thus, we use the following trick  borrowed from the calculus of variation theory. Let 
		\begin{equation*} \alpha(r)\triangleq \begin{cases}
				{(1-|r|)^2}/4, & |r|\leq 1\\
				0, & \text{otherwise}
		\end{cases}\end{equation*}  Given natural $n$, let $\epsilon_n\in (0,1/2)$ be such that, for each $r\in [R_n-\epsilon_n,R_n+\epsilon_n]$, \begin{equation}\label{C:ineq:A_n} A_n\vee A_{n-1}\leq 3g_0(r)/(4r).\end{equation}  The existence of such $\epsilon_n$ follows from the fact that $A_n\vee A_{n-1}\leq g_0(R_n)/2R_n$ (see~\eqref{B:ineq:A}).
		Notice that $|A_n-A_{n-1}|\leq 1\leq g_0(r)$.  Moreover, the intervals $[R_n-\epsilon_n,R_n+\epsilon_n]$ do not intersect.
		
		Now we set 
		\begin{equation*} g(r)\triangleq g_0(r)+\sum_{n=1}^\infty \epsilon_n(A_n-A_{n-1})\alpha\Bigg(\frac{r-R_n}{\epsilon_n}\Bigg).\end{equation*}  Notice that,  $g(r)=g_0(r)$ outside the intervals $[R_n-\epsilon_n,R_n+\epsilon_n]$; whereas, on each interval $[R_n-\epsilon_n,R_n+\epsilon_n]$, $g(r)=g_0(r)+\epsilon_n(A_n-A_{n-1})\alpha((r-R_n)/\epsilon_n)$. Since each $A_n\in [0,1]$, $\epsilon_n\in (0,1/2)$, from~\eqref{B:ineq:int} we have that 
		\begin{equation*}\label{B:ineq:int_g_fin}
			\begin{split}
				\int_{\rd}\mathscr{a}^2&(x)g(\mathscr{a}(x))m(dx)\\&\leq \int_{\rd}\mathscr{a}^2(x)g_0(\mathscr{a}(x))m(dx)+\frac{1}{8}\int_{\rd}\mathscr{a}^2(x)m(dx)<\infty.\end{split}
		\end{equation*} This proves \ref{B:cond:G_int}.

		The very definition of the function $g$ gives that the function $g$ is continuously differentiable on  $(0,+\infty)\setminus \cup_{n=1}^\infty [R_n-\epsilon_n,R_n+\epsilon_n]$. Moreover, for $r\in (R_{n-1}+\epsilon_{n-1},R_n-\epsilon_n)$,  $g'(r)=g_0'(r)=A_{n-1}$.  If $r\in [R_n-\epsilon_n,R_n+\epsilon_n]$, then direct calculations gives that 
		\begin{equation}\label{B:equality_n}
			g'(r)=(A_n+A_{n-1})/2+\frac{r-R_n}{2\epsilon_n}(A_n-A_{n-1}).
		\end{equation} This gives that $g'(\cdot)$ is continuous on $[0,+\infty)$ and positive on $[R_1-\epsilon_1,+\infty)$. The fact that $g(r)\rightarrow\infty$ as $r\rightarrow\infty$ follows from the definitions of the functions $g$ and $g_0$. Thus, we obtain~\ref{B:cond:G_positive}.
		
		
		It remains to show that condition~\ref{B:cond:G_derivative} holds true. When $r\in (0,+\infty)\setminus \cup_{n=1}^\infty [R_n-\epsilon_n,R_n+\epsilon_n]$, we have (see \eqref{B:ineq:A}) that $g'(r)\leq g(r)/(2r)\leq g(r)/r$. Notice that, on $[R_n-\epsilon_n,R_n+\epsilon_n]$, $g'(r)\leq A_n\vee A_{n-1}$ (see~\eqref{B:equality_n}). Using the inequality \(|g(r)-g_0(r)|\leq 1/4\) and estimate~\eqref{C:ineq:A_n}, we conclude that
		\begin{equation*}
			g'(r)\leq A_n\vee A_{n-1}\leq \frac{3g_0(r)}{4r}\leq \frac{g(r)}{r}.
		\end{equation*}  This completes the proof.
	\end{proof}	
	
	\begin{lemma}\label{B:lm:growth} If a function $g:(0,+\infty)\rightarrow (0,+\infty)$ satisfies conditions~\ref{B:cond:G_positive} and~\ref{B:cond:G_derivative}, then, for every $r\in (0,+\infty)$, 
		\begin{equation*} g(2r)\leq (1-\ln 2)^{-1}g(r).\end{equation*}  
	\end{lemma}
	\begin{proof}
		Due to \ref{B:cond:G_derivative}, we have that
		\begin{equation*} g(2r)-g(r)=\int_0^1g'(r+hr)rdh\leq \int_{0}^1\frac{g(r+hr)}{1+h}dh.\end{equation*}  Condition \ref{B:cond:G_positive} gives that $g(r+hr)\leq g(2r)$. The  statement of the lemma directly follows from this.	
	\end{proof}

	\begin{proof}[Proof of Lemma~\ref{B:lm:derivative_g}]
		For simplicity, we denote
		\begin{equation*} \varphi(x)\triangleq \mathscr{a}^2(x)g(\mathscr{a}(x)).\end{equation*}  Obviously,
		\begin{equation*}\label{B:equality:nabla_phi}\nabla\varphi(x)=\mathscr{j}_g(x).\end{equation*} Thanks to the definition of the function $\mathscr{a}$ (see~\eqref{intro:almost_norm}) and condition~\ref{B:cond:G_derivative}, we have that $\mathscr{j}_g\in\ID{m}$. 	
		{ Moreover, property~\ref{B:cond:G_int} implies that $V_g(m)$ is finite. Furthermore, if $C>0$, $F'\in\SL{m}$ is such that $\|F'(x)\|\leq C\mathscr{a}(x)$ $m$-a.e., we have that
			\begin{equation*} \mathscr{a}^2(x+hF'(x))\leq 2\mathscr{a}^2(x)+2h^2C^2\mathscr{a}^2(x)\ \ m\text{-a.e.}\end{equation*}  Hence, for sufficiently small $h$, we have that, for $m$-a.e. $x\in\rd$,
			\begin{equation*} \mathscr{a}(x+hF'(x))\leq 2\mathscr{a}(x).\end{equation*}  This estimate is uniform over the class of functions $F'$ satisfying \(\|F'(x)\|\leq C\mathscr{a}(x)\) for \(m\)-a.e. \(x\in\rd\) and some constant $C$. 
			
			From this and Lemma~\ref{B:lm:growth}, we have that $V_g((\operatorname{Id}+hF')\sharp m)$ is uniformly bounded whenever $h$ is sufficiently small.}

		{	Now let \(F\in\SL{m}\) and let \(\{(h_n,F_n)\}_{n=1}^\infty\in\VN{m}{F}\). In particular, there exists a constant \(C>0\) such that \(\|F_n(x)\|\leq C\mathscr{a}(x)\) for all \(n\) and \(m\)-a.e. \(x\in\rd\). Since \(h_n\to0\), we may assume without loss of generality that \(h_n\in(0,C^{-1})\). For each $x\in\rd$, we have that 
			\begin{equation*} \varphi(x+h_n F_n(x))-\varphi(x)=\int_0^1\langle \mathscr{j}_g(x+rh_n F_n(x)),h_n F_n(x)\rangle dr.\end{equation*}  Thanks to the fact that $\|\mathscr{j}_g(x)\|\leq 3g(\mathscr{a}(x)){\mathscr{a}(x)}$ and the choice of the sequence $\{F_n\}_{n=1}^\infty$, we have that the function $x\mapsto \int_0^1\langle \mathscr{j}_g(x+rh_n F_n(x)), F_n(x)\rangle dr$ is bounded by some  function that is integrable w.r.t. the measure $m$. Moreover,  for each $x$,  $\int_0^1\langle \mathscr{j}_g(x+rh_n F_n(x)), F_n(x)\rangle dr\rightarrow \langle \mathscr{j}_g(x),F(x)\rangle $ as $n\rightarrow\infty$. Thus, \begin{equation*} \int_{\rd}\int_0^1\langle \mathscr{j}_g(x+rh_n F_n(x)), F_n(x)\rangle dr m(dx)\rightarrow \int_{\rd}\langle \mathscr{j}_g(x), F(x)\rangle m(dx)\text{ as }n\rightarrow \infty.\end{equation*}  This gives that
			\begin{equation*} \lim_{n\rightarrow \infty}\frac{\varphi(x+h_n F_n(x))-\varphi(x)-\langle \mathscr{j}_g,F\rangle_m}{h_n}=0.\end{equation*} 
			Using the definition of the sub- and superdifferentials, we obtain the statement of the lemma.}
	\end{proof}
		
\bibliography{references}
\end{document}